\documentclass[10pt,a4paper]{article}
\pdfoutput=1    

\usepackage{amsmath}
\usepackage{amsfonts}
\usepackage[mathscr]{eucal}
\usepackage{amsthm}
\usepackage{amssymb}
\usepackage{subfigure}
\usepackage{cite}
\usepackage{color}

\DeclareMathOperator{\supp}{supp}
\DeclareMathOperator{\dist}{dist}
\usepackage{pgf,tikz}\usetikzlibrary{matrix, calc, arrows}
\usepackage[bookmarks=true]{hyperref}

\definecolor{myblue}{rgb}{0,0,0.6}
\hypersetup{pdftitle={Interpolation of Hilbert and Sobolev Spaces: Quantitative Estimates and Counterexamples},
            pdfauthor={Chandler-Wilde, Hewett, Moiola},
     colorlinks=true, linkcolor=myblue,  citecolor=myblue, filecolor=myblue,   urlcolor=myblue,  }
\allowdisplaybreaks[4]

\begin{document}
\newcommand{\hs}[1]{\hspace{#1mm}}
\newcommand{\rd}{\mathrm{d}}
\newcommand{\R}{\mathbb{R}}
\newcommand{\N}{\mathbb{N}}
\newcommand{\Z}{\mathbb{Z}}
\newcommand{\C}{\mathbb{C}}
\newcommand{\ri}{{\mathrm{i}}}
\newcommand{\re}{{\mathrm{e}}}
\newcommand{\scrD}{\mathscr{D}}
\newcommand{\scrS}{\mathscr{S}}
\newcommand{\tH}{\widetilde{H}}

\newtheorem{thm}{Theorem}[section]
\newtheorem{lem}[thm]{Lemma}
\newtheorem{defn}[thm]{Definition}
\newtheorem{prop}[thm]{Proposition}
\newtheorem{cor}[thm]{Corollary}
\newtheorem{rem}[thm]{Remark}
\newtheorem{example}[thm]{Example}

\title{Interpolation of Hilbert and Sobolev Spaces:\\ Quantitative Estimates and Counterexamples}
\author{S.\ N.\ Chandler-Wilde\footnotemark[1], D.\ P.\ Hewett\footnotemark[2], A.\ Moiola\footnotemark[1]}

\renewcommand{\thefootnote}{\fnsymbol{footnote}}
\footnotetext[1]{Department of Mathematics and Statistics, University of Reading, Whiteknights PO Box 220, Reading RG6 6AX, UK. Email: \texttt{s.n.chandler-wilde@reading.ac.uk, a.moiola@reading.ac.uk}.}
\footnotetext[2]{Mathematical Institute, University of Oxford, Radcliffe Observatory Quarter, Woodstock Road, Oxford, OX2 6GG, UK. Email: \texttt{hewett@maths.ox.ac.uk}}

\maketitle

\begin{center}
{\em Dedicated to Vladimir Maz'ya, on the occasion of his 75th Birthday}
\end{center}

\begin{abstract}
This paper provides an overview of interpolation of Banach and Hilbert spaces, with a focus on establishing when equivalence of norms is in fact equality of norms in the key results of the theory. (In brief, our conclusion for the Hilbert space case is that, with the right normalisations, all the key results hold with equality of norms.) In the final section we apply the Hilbert space results to the Sobolev spaces $H^s(\Omega)$ and $\widetilde{H}^s(\Omega)$, for $s\in \mathbb{R}$ and an open $\Omega\subset \mathbb{R}^n$. We exhibit examples in one and two dimensions of sets $\Omega$ for which these scales of Sobolev spaces are not interpolation scales. In the cases when they {\em are} interpolation scales (in particular, if $\Omega$ is Lipschitz) we exhibit examples that show that, in general, the interpolation norm does not coincide with the intrinsic Sobolev norm and, in fact, the ratio of these two norms can be arbitrarily large.
\end{abstract}

\section{Introduction} \label{sec:intro}

This paper\footnotemark[3]\footnotetext[3]{This is the version of the paper as published in {\em Mathematika,} {\bf 61} (2015) 414--443. {\bf Please read this version alongside the Corrigendum to the published paper, which has been accepted for publication in {\em Mathematika} and is reproduced at the end of this document.}} provides in the first two sections a self-contained overview of the key results of the real method of interpolation for Banach and Hilbert spaces. This is a classical subject of study (see, e.g., \cite{BeLo,Triebel78ITFSDO,BeSh,Tartar} and the recent review paper \cite{Ameur2014} for the Hilbert space case), and it might be thought that there is little more to be said on the subject. The novelty of our presentation---this the perspective of numerical analysts who, as users of interpolation theory, are ultimately concerned with the computation of interpolation norms and the computation of error estimates expressed in terms of interpolation norms---is that we pay particular attention to the question:
``When is equivalence of norms in fact equality of norms in the interpolation of Banach and Hilbert spaces?''

\renewcommand{\thefootnote}{\arabic{footnote}}

At the heart of the paper is the study, in Section \ref{sec:interpHil}, of the interpolation of Hilbert spaces $H_0$ and $H_1$ embedded in a larger linear space $V$, in the case when the interpolating space is also Hilbert (this the so-called problem of {\em quadratic interpolation}, see, e.g., \cite{Donoghue,LiMaI,McCarthy,Ameur2004,Ameur2014}). The one line summary of this section is that all the key results of interpolation theory hold with ``equality of norms'' in place of ``equivalence of norms'' in this Hilbert space case, and this with minimal assumptions, in particular we assume nowhere that our Hilbert spaces are separable (as, e.g., in \cite{LiMaI,McCarthy,Ameur2004,Ameur2014}).

Real interpolation between Hilbert spaces $H_0$ and $H_1$ produces interpolation spaces $H_\theta$, $0<\theta<1$,  intermediate between $H_0$ and $H_1$. In the last section of the paper we apply the Hilbert space interpolation results of \S\ref{sec:interpHil} to the  Sobolev spaces $H^s(\Omega):= \{U|_\Omega: U\in H^s(\R^n)\}$ and $\tH^s(\Omega)$ (defined as the closure of $C_0^\infty(\Omega)$ in $H^s(\R^n)$), for $s\in \R$.  Questions we address are:
\begin{enumerate}
\item[(i)] For what ranges of $s$ are $H^s(\Omega)$ and $\tH^s(\Omega)$ interpolation scales, meaning that the interpolation space $H_\theta$, when interpolating between $s=s_0$ and $s=s_1$, is the expected intermediate Sobolev space with $s=s_0(1-\theta)+s_1\theta$?
\item[(ii)] When the interpolation space is the expected intermediate Sobolev space, do the interpolation space norm and intrinsic Sobolev norm coincide (the interpolation scale is exact), or, if not, how different can they be?
\end{enumerate}
A main result of the paper is to exhibit one- and two-dimensional counterexamples that show that $H^s(\Omega)$ and $\tH^s(\Omega)$ are not in general interpolation scales. It is well-known that these Sobolev spaces {\em are} interpolation scales for all $s\in \R$ when $\Omega$ is Lipschitz. In that case we demonstrate, via a number of counterexamples that, in general (we suspect, in fact, whenever $\Omega \subsetneqq \R^n$), $H^s(\Omega)$ and $\tH^s(\Omega)$ are not exact interpolation scales. Indeed, we exhibit simple examples where the ratio of interpolation norm to intrinsic Sobolev norm may be arbitrarily large.  Along the way we give explicit formulas for some of the interpolation norms arising that may be of interest in their own right.
We remark that our investigations, which are inspired by applications arising in boundary integral equation methods (see \cite{ChaHewMoi:13}), in particular are inspired by McLean \cite{McLean}, and by its appendix on interpolation of Banach and Sobolev spaces. However a result of \S\ref{sec:sob} is that one result claimed by McLean (\!\cite[Theorem B.8]{McLean}) is false.

Much of the Hilbert space Section \ref{sec:interpHil} builds strongly on previous work. In particular, our result that, with the right normalisations, the norms in the $K$- and $J$-methods of interpolation coincide in the Hilbert space case is a (corrected version of) an earlier result of Ameur \cite{Ameur2004} (the normalisations proposed and the definition of the $J$-method norm seem inaccurate in \cite{Ameur2004}). What is new in our Theorem \ref{thm:JequalK} is the method of proof---all of our proofs in this section are based on the spectral theorem that every bounded normal operator is unitarily equivalent to a multiplication operator on $L^2(\mathcal{X},M,\mu)$, for some measure space  $(\mathcal{X},M,\mu)$, this coupled with an elementary explicit treatment of interpolation on weighted $L^2$ spaces---which deals seamlessly with the general Hilbert space case without an assumption of separability or that $H_0\cap H_1$ is dense in $H_0$ and $H_1$. Again, our result in Theorem \ref{thm:unique} that
there is only one (geometric) interpolation space of exponent $\theta$, when interpolating Hilbert spaces, is a version of McCarthy's  \cite{McCarthy} uniqueness theorem. What is new is that we treat the general Hilbert space case by a method of proof based on the aforementioned spectral theorem. Our focus in this section is real interpolation, but we note in Remark \ref{rem:functor} that, as a consequence of this uniqueness result (as noted in \cite{McCarthy}), complex and real interpolation coincide in this Hilbert space case.

While our focus is primarily on interpolation of Hilbert spaces, large parts of the theory of interpolation spaces are appropriately described in the more general
Banach space context, not least when trying to clarify those results independent of the extra structure a Hilbert space brings. The first \S\ref{sec:interp} describes real interpolation in this general Banach space context. Mainly this section sets the scene. What is new is that our perspective leads us to pay close attention to the precise choice of normalisation in the definitions of the $K$- and $J$-methods of real interpolation (while at the same time making definitions suited to the later Hilbert space case).

We intend primarily that, throughout, vector space, Banach space, Hilbert space, should be read as {\em complex} vector space, Banach space, Hilbert space. But all the definitions and results proved apply equally in the real case with fairly obvious and minor changes and extensions to the arguments.

We finish this introduction by a few words on the history of interpolation (and see \cite{BeLo,Triebel78ITFSDO,BeSh,Tartar}).
There are two standard procedures for constructing interpolation spaces (see, e.g., \cite{BeLo}) in the Banach space setting. The first is the {\em complex method} due to Lions and Calder\'on, this two closely-related procedures for constructing interpolation spaces \cite[Section 4.1]{BeLo}, inspired by the classical proof of the Riesz--Thorin interpolation theorem. (These two procedures applied to a compatible pair $\overline X =(X_0,X_1)$ (defined in \S\ref{sec:interp}) produce the identical Banach space (with the identical norm) if either one of $X_0$ or $X_1$ is reflexive, in particular if either is a Hilbert space \cite[Theorem 4.3.1]{BeLo}.) We will mention the complex method only briefly, in Remark~\ref{rem:functor}. Our focus is on the so-called {\em real interpolation method}. This term is used to denote a large class of methods for  constructing interpolation spaces from a compatible pair,
all these methods constructing the same interpolation spaces \cite{Triebel78ITFSDO} (to within isomorphism, see Theorem \ref{JKsame} below).
In this paper we focus on the two standard such methods, the {\em K-method} and the {\em J-method}, which are complementary, dual constructions due to Peetre and Lions (see e.g.\ \cite{Peetre}),  inspired by the classical Marcinkiewicz interpolation theorem \cite[Section 1.3]{BeLo}.

\section{Real interpolation of Banach spaces} \label{sec:interp}

Suppose that $X_0$ and $X_1$ are Banach spaces that are linear subspaces of some larger vector space $V$. In this case we say that $\overline X = (X_0,X_1)$ is a {\em compatible pair} and $\Delta = \Delta(\overline X) := X_0\cap X_1$ and $\Sigma =\Sigma(\overline X) := X_0+X_1$ are also linear subspaces of $V$: we equip these subspaces with the norms
\begin{equation*}
\|\phi\|_{\Delta} := \max\big\{ \|\phi\|_{X_0},\|\phi\|_{X_1} \big\}
\end{equation*}
and
\begin{equation*} 
\|\phi\|_{\Sigma} := \inf\big\{\|\phi_0\|_{X_0} + \|\phi_1\|_{X_1} :\phi_0\in X_0, \, \phi_1\in X_1, \, \phi_0+\phi_1=\phi\big\},
\end{equation*}
with which $\Delta$ and $\Sigma$ are Banach spaces \cite[Lemma 2.3.1]{BeLo}. 
We note that, for $j=0,1$,
$
\Delta \subset X_j\subset \Sigma
$,
and these inclusions are continuous as $\|\phi\|_{\Sigma} \leq \|\phi\|_{X_j}$, $\phi \in X_j$, and $\|\phi\|_{X_j} \leq \|\phi\|_{\Delta}$, $\phi\in \Delta$. Thus every compatible pair is a pair of Banach spaces that are subspaces of, and continuously embedded in, a larger Banach space. In our later application to Sobolev spaces we will be interested in the important special case where $X_1\subset X_0$. In this case $\Delta= X_1$ and $\Sigma=X_0$ with equivalence of norms, indeed equality of norms if $\|\phi\|_{X_1} \geq \|\phi\|_{X_0}$, for $\phi \in X_1$.

If $X$ and $Y$ are Banach spaces and $B:X\to Y$ is a bounded linear map, we will denote the norm of $B$ by $\|B\|_{X,Y}$, abbreviated as $\|B\|_X$ when $X=Y$. Given compatible pairs $\overline X = (X_0,X_1)$ and $\overline Y = (Y_0,Y_1)$ one calls the linear map $A:\Sigma(\overline X)\to \Sigma(\overline Y)$ a {\em couple map}, and writes $A:\overline X\to \overline Y$, if $A_j$, the restriction of $A$ to $X_j$, is a bounded linear map from $X_j$ to $Y_j$. Automatically $A:\Sigma(\overline X)\to\Sigma(\overline Y)$ is bounded and $A_\Delta$, the restriction of $A$ to $\Delta(\overline X)$, is also a bounded linear map from $\Delta(\overline X)$ to $\Delta(\overline Y)$. On the other hand, given bounded linear operators $A_j:X_j\to Y_j$, for $j=0,1$, one says that $A_0$ and $A_1$ are {\em compatible} if $A_0\phi = A_1\phi$, for $\phi\in \Delta(\overline X)$. If $A_0$ and $A_1$ are compatible then there exists a unique couple map $A:\Sigma(\overline X)\to \Sigma(\overline Y)$ which has $A_0$ and $A_1$ as its
restrictions to $X_0$ and $X_1$, respectively.

 Given a compatible pair $\overline X=(X_0,X_1)$ we will call a Banach space $X$ an {\em intermediate space} between $X_0$ and $X_1$  \cite{BeLo} if $\Delta\subset X\subset \Sigma$ with continuous inclusions. We will call an intermediate space $X$ an {\em interpolation space} relative to $\overline X$ if, whenever $A:\overline X\to \overline X$, it holds that $A(X) \subset X$ and $A:X\to X$ is a bounded linear operator. Generalising this notion, given compatible pairs $\overline X$ and $\overline Y$, and Banach spaces $X$ and $Y$, we will call $(X,Y)$  {\em a pair of interpolation spaces relative to $(\overline X, \overline Y)$} if $X$ and $Y$ are intermediate with respect to $\overline X$ and $\overline Y$, respectively, and if, whenever $A:\overline X\to \overline Y$, it holds that $A(X)\subset Y$ and $A:X\to Y$ is a bounded linear operator \cite{BeLo}. If $(X,Y)$ is a pair of interpolation spaces relative to $(\overline X, \overline Y)$ then \cite[Theorem~2.4.2]{BeLo} there exists $C>0$ such that, whenever
$A:\overline X\to \overline Y$, it holds that
 \begin{equation} \label{eq:is}
 \|A\|_{X,Y} \leq C \max\big\{\|A\|_{X_0,Y_0},\|A\|_{X_1,Y_1}\big\}.
 \end{equation}
 If the bound \eqref{eq:is} holds for every $A:\overline X\to \overline Y$ with $C=1$, then $(X,Y)$ are said to be {\em exact} interpolation spaces: for example the pairs $(\Delta(\overline X), \Delta(\overline Y))$ and $(\Sigma(\overline X), \Sigma(\overline Y))$ are exact interpolation spaces with respect to $(\overline X,\overline Y)$, for all compatible pairs $\overline X$ and $\overline Y$ \cite[Section~2.3]{BeLo}. If, for all $A:\overline X\to\overline Y$,
\begin{equation} \label{eq:is1}
 \|A\|_{X,Y} \leq \|A\|^{1-\theta}_{X_0,Y_0}\,\|A\|^\theta_{X_1,Y_1},
 \end{equation}
 then the interpolation space pair $(X,Y)$ is said to be {\em exact of exponent } $\theta$.

\subsection{The \texorpdfstring{$K$}{K}-method for real interpolation}
\label{sec:K}

To explain the $K$-method, for every compatible pair $\overline X=(X_0,X_1)$ define the  $K${\em-functional} by
\begin{equation} \label{eq:Kfunct}
K(t,\phi) = K(t,\phi, \overline X) := \inf \left\{\left(\|\phi_0\|^2_{X_0} + t^2 \|\phi_1\|^2_{X_1}\right)^{1/2}: \phi_0\in X_0, \, \phi_1\in X_1, \, \phi_0+\phi_1=\phi \right\},
\end{equation}
for $t>0$ and $\phi\in \Sigma(\overline X)$; our definition is precisely that of \cite[p.~98]{LiMaI}, \cite{Bramble,McLean}. (More usual, less suited to the Hilbert space case, but leading to the same interpolation spaces and equivalent norms, is to replace the 2-norm $\left(\|\phi_0|^2_{X_0} + t^2 \|\phi_1\|^2_{X_1}\right)^{1/2}$ by the 1-norm $\|\phi_0\|_{X_0} + t \|\phi_1\|_{X_1}$ in this definition, e.g.~\cite{BeLo}.) Elementary properties of this $K$-functional are noted in \cite[p.~319]{McLean}. An additional elementary calculation is that, for $\phi\in \Delta$,
\begin{equation} \label{K1boundsK}
K(t,\phi) \leq K_1(t,\phi) := \inf_{a\in \C} \left(|a|^2\|\phi\|^2_{X_0} + t^2|1-a|^2 \|\phi\|^2_{X_1}\right)^{1/2} = \frac{t\|\phi\|_{X_0}\|\phi\|_{X_1}}{\left(\|\phi\|_{X_0}^2 + t^2\|\phi\|_{X_1}^2\right)^{1/2}\,}\, ,
\end{equation}
this infimum achieved by the choice
$a = t^2\|\phi\|_{X_1}^2/(\|\phi\|_{X_0}^2 + t^2\|\phi\|_{X_1}^2)$.

Next we define a weighted $L^q$ norm by
$$
\|f\|_{\theta,q} := \left(\int_0^\infty |t^{-\theta} f(t)|^q  \,\frac{\rd t}{t}\right)^{1/q}, \quad \mbox{for } 0<\theta<1 \mbox{ and } 1\leq q<\infty,
$$
with the modification when $q=\infty$, that
\begin{equation} \label{eqInfCase}
\|f\|_{\theta,\infty} := \mathrm{ess}\,\sup_{t>0} |t^{-\theta} f(t)|.
\end{equation}
Now define, for every compatible pair $\overline X=(X_0,X_1)$,  and for $0<\theta<1$ and $1\leq q\leq \infty$,
\begin{equation} \label{defKnorm}
K_{\theta,q}(\overline X) := \big\{\phi\in \Sigma(\overline X): \|K(\cdot,\phi)\|_{\theta,q} <\infty\big\},
\end{equation}
this a normed space (indeed a Banach space \cite[Theorem 3.4.2]{BeLo}) with the norm
\begin{equation} \label{eq:normK}
\|\phi\|_{K_{\theta,q}(\overline X)} := N_{\theta,q} \|K(\cdot,\phi)\|_{\theta,q}.
\end{equation}
Here the constant $N_{\theta,q}>0$ is an arbitrary normalisation factor. We can, of course, make the (usual) choice $N_{\theta,q}=1$, but our preferred choice of $N_{\theta,q}$ will be, where $g(s) := s/\sqrt{1+s^2}$,
\begin{equation} \label{Ntheta}
N_{\theta,q} := \|g\|^{-1}_{\theta,q} = \left\{\begin{array}{cc}
                         \left(\int_0^\infty s^{q(1-\theta)-1}(1+s^2)^{-q/2}  \,\rd s\right)^{-1/q}, & 1\leq q<\infty, \\
                         \theta^{-\theta/2}(1-\theta)^{-(1-\theta)/2}, & q=\infty;
                       \end{array}\right.
\end{equation}
the supremum in \eqref{eqInfCase} when $f=g$ achieved for $t=\sqrt{(1-\theta)/\theta}$. We note that, with this choice, $N_{\theta,q}=N_{1-\theta,q}$ (substitute $s=t^{-1}$ in \eqref{Ntheta}). Further, $\min(1,s)/\sqrt{2}\leq g(s)\leq \min(1,s)$, so that $N^\prime_{\theta,q}\leq N_{\theta,q}\leq\sqrt{2}\,N_{\theta,q}^\prime$, where
$$
N^\prime_{\theta,q} := \|\min(1,\cdot)\|^{-1}_{\theta,q} = \left\{\begin{array}{cc}
                         [q\theta(1-\theta)]^{1/q}, & 1\leq q<\infty, \\
                         1, & q=\infty.
                       \end{array}\right.
$$
We note also that \cite[Exercise B.5]{McLean}, with the choice \eqref{Ntheta},
\begin{equation} \label{eq:Nchoice}
N_{\theta,2} = \big((2/\pi)\sin(\pi\theta)\big)^{1/2}.
\end{equation}
The normalisation $N^\prime_{\theta,q}$ is used in \cite[(B.4)]{McLean} (and see \cite[Theorem 3.4.1(e)]{BeLo}); \eqref{eq:Nchoice} is used in \cite[(B.9)]{McLean}, \cite[p.~143]{Bramble}, and dates back at least to \cite[p.~99]{LiMaI}.

$K_{\theta,q}(\overline X)$, for $0<\theta<1$ and $1\leq q\leq \infty$,  is the family of spaces constructed by the $K$-method. We will often use the alternative notation $(X_0,X_1)_{\theta,q}$ for $K_{\theta,q}(\overline X)$.

Our preference for the normalisation \eqref{Ntheta} is explained by part (iii) of the following lemma.
\begin{lem} \label{intnormub}
Suppose that $\overline X=(X_0,X_1)$ is a compatible pair and define the norm on $K_{\theta,q}(\overline X)$ with the normalisation \eqref{Ntheta}.
\begin{enumerate}
\item[(i)] If $\phi\in \Delta(\overline X)$ then $\phi\in K_{\theta,q}(\overline X)$ and
$\displaystyle{
\|\phi\|_{K_{\theta,q}(X)} \leq \|\phi\|_{X_0}^{1-\theta} \, \|\phi\|_{X_1}^\theta \leq \|\phi\|_{\Delta(\overline X)}
}$.

\item[(ii)]  If $\phi\in K_{\theta,q}(\overline X)$ then $\phi\in \Sigma(\overline X)$ and
$
\|\phi\|_{\Sigma(\overline X)} \leq \|\phi\|_{K_{\theta,q}(X)}.
$

\item[(iii)] If $X_0=X_1$ (with equality of norms) then $X_0=X_1=\Sigma(\overline X)=\Delta(\overline X) = K_{\theta,q}(\overline X)$, with equality of norms.
\end{enumerate}
\end{lem}
\begin{proof}
If $\phi\in \Delta(\overline X)$ is non-zero then, for $0<\theta<1$, $1\leq q<\infty$, using \eqref{K1boundsK},
$$
\|\phi\|^q_{K_{\theta,q}(X)} \leq N^q_{\theta,q} \|K_1(\cdot,\phi)\|_{\theta,q}^q = N^q_{\theta,q} \|\phi\|^q_{X_0}\|\phi\|^q_{X_1}\int_0^\infty \left[\frac{t^{1-\theta}}{\left(\|\phi\|_{X_0}^2 + t^2\|\phi\|_{X_1}^2\right)^{1/2}}\right]^q \,\frac{\rd t}{t} = \|\phi\|_{X_0}^{q(1-\theta)} \, \|\phi\|_{X_1}^{q\theta}\, ,
$$
the last equality a consequence of the identity
\begin{equation} \label{useful}
\int_0^\infty \frac{t^\alpha}{(a+bt^2)^{q/2}}\, \rd t
= a^{(\alpha+1-q)/2} \,b^{-(1+\alpha)/2} \int_0^\infty \frac{t^\alpha}{(1+t^2)^{q/2}}\, \rd t
=a^{(\alpha+1-q)/2} \,b^{-(1+\alpha)/2} N_{(q-\alpha-1)/q,q}^{-q},
\end{equation}
for $a,b>0$ and $-1<\alpha<q-1$.
Similarly,
$$
\|\phi\|_{K_{\theta,\infty}(X)} \leq N_{\theta,\infty} \|K_1(\cdot,\phi)\|_{\theta,\infty} = N_{\theta,\infty}  \|\phi\|_{X_0}^{1-\theta} \, \|\phi\|_{X_1}^{\theta} \sup_{s>0} \frac{s^{1-\theta}}{\sqrt{1+s^2}} = \|\phi\|_{X_0}^{1-\theta} \, \|\phi\|_{X_1}^{\theta}\, .
$$
Clearly also $\|\phi\|_{X_0}^{1-\theta} \, \|\phi\|_{X_1}^\theta \leq \|\phi\|_{\Delta(\overline X)}$ so that (i) holds.

For $\phi_0\in X_0$, $\phi_1\in X_1$, $\|\phi_0\|_{X_0}^2 + t^2 \|\phi_1\|_{X_1}^2 \geq (t^2/(1+t^2))(\|\phi_0\|_{X_0} + \|\phi_1\|_{X_1})^2$, from which it follows that
$$
K(t,\phi) \geq g(t) \|\phi\|_{\Sigma(\overline X)}, \quad \mbox{for } \phi\in \Sigma(\overline X), \; t>0,
$$
where $g(t)=t/\sqrt{1+t^2}$, which implies (ii).

To see (iii), we recall that we have observed already that, if $X_1\subset X_0$, with $\|\phi\|_{X_0}\leq \|\phi\|_{X_1}$, then $X_1=\Delta(\overline X)$ and $X_0=\Sigma(\overline X)$, with equality of norms. Thus (iii) follows from (i) and (ii).
\end{proof}

 The following theorem collects key properties of the spaces $K_{\theta,q}(\overline X)$, in the first place that they are indeed interpolation spaces. Of these properties: (i) is proved, for example, as \cite[Theorem B.2]{McLean}; (ii) in \cite[Theorem 3.4.1]{BeLo}; (iii) follows immediately from the definitions and Lemma \ref{intnormub}; (iv) and (v) are part of \cite[Theorem 3.4.2]{BeLo}; (vi) is obvious from the  definitions. Finally (vii) is the {\em reiteration} or {\em stability} theorem, that $K$-method interpolation of $K$-interpolation spaces gives the expected $K$-interpolation spaces, proved, for example, in \cite[Theorem B.6]{McLean}.
\begin{thm} \label{thm:interp} Suppose that $\overline X=(X_0,X_1)$ and $\overline Y=(Y_0,Y_1)$ are compatible pairs.  Then we have the following statements:
\begin{enumerate}
\item[(i)] For $0<\theta<1$, $1\leq q\leq \infty$, $(K_{\theta,q}(\overline X),K_{\theta,q}(\overline Y))$ is a pair of interpolation spaces with respect to $(\overline X,\overline Y)$ that is exact of exponent $\theta$.

\item[(ii)] For $0<\theta<1$, $1\leq q\leq \infty$, $(X_0,X_1)_{\theta,q} = (X_1,X_0)_{1-\theta,q}$, with equality of norms if $N_{\theta,q}=N_{1-\theta,q}$ (which holds for the choice \eqref{Ntheta}).
\item[(iii)] For $0<\theta_1<\theta_2<1$ and $1\leq q\leq \infty$, if $X_1\subset X_0$, then $X_1 \subset K_{\theta_2,q}(\overline X)\subset K_{\theta_1,q}(\overline X) \subset X_0$, and the inclusion mappings are continuous. Furthermore, if $\|\phi\|_{X_0}\leq \|\phi\|_{X_1}$, for $\phi\in X_1$, then, with the choice of normalisation \eqref{Ntheta}, $\|\phi\|_{K_{\theta_1,q}(\overline X)}\leq \|\phi\|_{K_{\theta_2,q}(\overline X)}$ for $\phi\in K_{\theta_2,q}(\overline X)$,
$$
\|\phi\|_{X_0} \leq \|\phi\|_{K_{\theta_1,q}(\overline X)}, \; \mbox{ for } \phi \in K_{\theta_1,q}(\overline X), \quad \mbox{and} \quad  \|\phi\|_{K_{\theta_2,q}(\overline X)}\leq \|\phi\|_{X_1}, \; \mbox{ for } \phi \in X_1.
$$
\item[(iv)] For $0<\theta < 1$, $1\leq q<\infty$, $\Delta(\overline X)$ is dense in $K_{\theta,q}(\overline X)$.
\item[(v)] For $0<\theta < 1$, $1\leq q<\infty$, where $X_j^\circ$ denotes the closure of $\Delta(\overline X)$ in $X_j$,
$$
(X_0,X_1)_{\theta,q} = (X_0^\circ,X_1)_{\theta,q} = (X_0,X_1^\circ)_{\theta,q} = (X_0^\circ,X_1^\circ)_{\theta,q},
$$
with equality of norms.
\item[(vi)] For $0<\theta < 1$, $1\leq q\leq \infty$, if $Z_j$ is a closed subspace of $X_j$, for $j=0,1$, and $\overline Z=(Z_0,Z_1)$, then
$$
K_{\theta,q}(\overline Z) \subset K_{\theta,q}(\overline X), \quad \mbox{with} \quad \|\phi\|_{K_{\theta,q}(\overline X)} \leq  \|\phi\|_{K_{\theta,q}(\overline Z)}, \; \mbox{ for } \phi\in K_{\theta,q}(\overline Z).
$$
\item[(vii)] Suppose that $1\leq q\leq \infty$,  $\theta_0,\theta_1\in [0,1]$, and, for $j=0,1$, $Z_j := (X_0,X_1)_{\theta_j,q}$, if $0<\theta_j<1$, while $Z_j:= X_{\theta_j}$, if $\theta_j\in \{0,1\}$. Then $(Z_0,Z_1)_{\eta,q} = (X_0,X_1)_{\theta,q}$, with equivalence of norms,
 for  $\theta = (1-\eta)\theta_0 + \eta \theta_1$ and $0<\eta<1$.
\end{enumerate}
\end{thm}

\subsection{The  \texorpdfstring{$J$}{J}-method}
\label{sec:J}
We turn now to study of the $J$-method, which we will see is complementary and dual to the $K$-method. Given a compatible pair $\overline X = (X_0,X_1)$, define the {\em J-functional} by
$$
J(t,\phi) = J(t,\phi,\overline X) :=  \left(\|\phi\|^2_{X_0} + t^2 \|\phi\|^2_{X_1}\right)^{1/2}, \quad \mbox{for }t>0 \mbox{ and } \phi\in \Delta(\overline X),
$$
our precise definition here that of \cite{McLean}. (More usual, less suited to the Hilbert space case but leading to the same interpolation spaces and equivalent norms, is to define $J(t,\phi):=\max(\|\phi\|_{X_0},t \|\phi\|_{X_1})$, e.g.~\cite{BeLo}.)
The space $J_{\theta,q}(\overline X)$ is now defined as follows. The elements of $J_{\theta,q}(\overline X)$ are those $\phi\in \Sigma(\overline X)$ that can be represented in the form
\begin{equation} \label{eq:intcon}
\phi = \int_0^\infty f(t)\,\frac{\rd t}{t},
\end{equation}
for some function $f:(0,\infty)\to \Delta(\overline X)$ that is strongly $\Delta(\overline X)$-measurable (see, e.g., \cite[p.~321]{McLean}) when $(0,\infty)$ is equipped with Lebesgue measure, and such that
\begin{equation} \label{Jconstraint}
\int_a^b \|f(t)\|_{\Delta(\overline X)} \,\frac{\rd t}{t} < \infty \; \mbox{ if } 0<a<b<\infty, \quad \mbox{and }\int_0^\infty \|f(t)\|_{\Sigma(\overline X)} \,\frac{\rd t}{t} < \infty.
\end{equation}
$J_{\theta,q}(\overline X)$ is a normed space with the norm defined by
\begin{equation} \label{Jnormdef}
\|\phi\|_{J_{\theta,q}(\overline X)} := N^{-1}_{\theta,q^*}\inf_f \|g_f\|_{\theta,q},
\end{equation}
where
\begin{equation} \label{qstar}
\frac{1}{q^*}+\frac{1}{q} = 1,
\end{equation}
$g_f(t) := J(t,f(t))$, and the infimum is taken over all measurable $f:(0,\infty)\to \Delta(\overline X)$ such that \eqref{eq:intcon} and \eqref{Jconstraint} hold. Our definition is standard (precisely as in \cite{McLean}), except for the normalisation factor $N^{-1}_{\theta,q^*}$. It is a standard result, e.g.\ \cite[Theorem B.3]{McLean}, that the spaces $K_{\theta,q}(\overline X)$ and $J_{\theta,q}(\overline X)$ coincide.

\begin{thm} \label{JKsame} For $0<\theta < 1$, $1\leq q\leq \infty$, $J_{\theta,q}(\overline X)=K_{\theta,q}(\overline X)$, with equivalence of norms.
\end{thm}

A major motivation for introducing the $J$-method is the following duality result. Here, for a Banach space $X$, $X^*$ denotes the dual of $X$.

\begin{thm} \label{JKdual} If $\overline X=(X_0, X_1)$ is a compatible pair and $\Delta(\overline X)$ is dense in $X_0$ and $X_1$, then $\Delta(\overline X)$ is dense in $\Sigma(\overline X)$ and $\overline X^*:=(X_0^*,X_1^*)$ is a compatible pair, and moreover
\begin{equation} \label{JKdual1}
\Delta(\overline X)^* = \Sigma(\overline X^*) \quad \mbox{and} \quad \Sigma(\overline X)^* = \Delta(\overline X^*),
\end{equation}
with equality of norms.
Further, for $0<\theta<1$, $1\leq q< \infty$, with $q^*$ defined by \eqref{qstar},
$$
(X_0,X_1)_{\theta,q}^* = (X_0^*,X_1^*)_{\theta,q^*},
$$
with equivalence of norms: precisely, if we use the normalisation \eqref{Ntheta}, for $\phi\in (X_0,X_1)_{\theta,q}$,
$$
\|\phi\|_{K_{\theta,q}(\overline X)^*} \leq  \|\phi\|_{J_{\theta,q^*}(\overline X^*)} \quad \mbox{and} \quad \|\phi\|_{K_{\theta,q^*}(\overline X^*)} \leq  \|\phi\|_{J_{\theta,q}(\overline X)^*}.
$$
\end{thm}
\begin{proof} We embed $X_j^*$ in $\Delta(\overline X)^*$, for $j=0,1$, in the obvious way, mapping $\psi\in X_j^*$ to its restriction to $\Delta(\overline X)$, this mapping injective since $\Delta(\overline X)$ is dense in $X_j$. That \eqref{JKdual1} holds is shown as Theorem 2.7.1 in \cite{BeLo}.
The remainder of the theorem is shown in the proof of \cite[Theorem B.5]{McLean}.
\end{proof}

The above theorem has the following corollary that is one motivation for our choice of normalisation in \eqref{Jnormdef} (cf., the corresponding result for $K$-norms in Lemma 2.1 (iii)).
\begin{cor} \label{Jlemma}
If $\overline X=(X,X)$ then $J_{\theta,q}(\overline X)=X$ with equality of norms.
\end{cor}
\begin{proof} It is clear, from Lemma \ref{intnormub} and Theorem \ref{JKsame}, that $J_{\theta,q}(\overline X)=X$. It remains to show equality of the norms which we will deduce from Theorem \ref{JKdual} for $1<q\leq \infty$.

We first observe (cf. part (vi) of Theorem \ref{thm:interp}) that, for $0<\theta < 1$, $1\leq q\leq \infty$, it follows immediately from the definitions that if $Z_j$ is a closed subspace of $Y_j$, for $j=0,1$, and $\overline Z=(Z_0,Z_1)$, $\overline Y = (Y_0,Y_1)$, then $\|\phi\|_{J_{\theta,q}(\overline Y)} \leq  \|\phi\|_{J_{\theta,q}(\overline Z)}$, for $\phi\in J_{\theta,q}(\overline Z)$.
We will apply this result in the case where, for some Banach space $X$, $Z_j=X$ and $Y_j=X^{**}$, the second dual of $X$, for $j=0,1$. (Recall that $X$ is canonically and isometrically embedded as a closed subspace of $X^{**}$, this subspace the whole of $X^{**}$ if $X$ is reflexive.)

Now suppose that $0<\theta<1$ and $1<q\leq \infty$. Then, for every Banach space $X$, where $\overline X = (X,X)$ and $\overline X^* = (X^*,X^*)$, it holds by Lemma \ref{intnormub} that $X^*=K_{\theta,q^*}(\overline X^*)$ with equality of norms. Applying Theorem \ref{JKdual}  it holds for $\phi\in X$ that
$$
\|\phi\|_X = \|\phi\|_{X^{**}} =  \|\phi\|_{K_{\theta,q^*}(\overline X^*)^*} \leq \|\phi\|_{J_{\theta,q}(\overline X^{**})} \leq  \|\phi\|_{J_{\theta,q}(\overline X)}
$$
and, where $\langle \cdot, \cdot \rangle$ is the duality pairing on $X\times X^*$,
$$
\|\phi\|_{J_{\theta,q}(\overline X)} = \sup_{0\neq \psi\in X^*} \frac{|\langle \phi, \psi \rangle|}{\|\psi\|_{J_{\theta,q}(\overline X)^*}} \leq
\sup_{0\neq \psi\in X^*} \frac{|\langle \phi, \psi \rangle|}{\|\psi\|_{K_{\theta,q^*}(\overline X^*)}}
= \sup_{0\neq \psi\in X^*} \frac{|\langle \phi, \psi \rangle|}{\|\psi\|_{X^*}} = \|\phi\|_X.
$$
Thus, for $\phi\in X$, $\|\phi\|_{J_{\theta,q}(\overline X)} = \|\phi\|_X$ for $0<\theta< 1$ and $1<q\leq \infty$.

To see that this holds also for $q=1$ we note that, for $1\leq q < \infty$, $0<\theta<1$, and $\phi \in X$,
\begin{equation*} 
\|\phi\|_{J_{\theta,q}(\overline X)} = \inf_f \mathcal{J}_{\theta,q}(f), \mbox{ where } \mathcal{J}_{\theta,q}(f) := N^{-1}_{\theta,q^*}\left(\int_0^\infty \left(\frac{\|f(t)\|_X}{g_\theta(t)}\right)^q \frac{\rd t}{t}\right)^{1/q},
\end{equation*}
$g_\theta(t) := t^\theta/\sqrt{1+t^2}$, and the infimum is taken over all $f$ that satisfy \eqref{eq:intcon} with $\int_0^\infty (\|f(t)\|_X\,/t)\, \rd t < \infty$. Note that $g_\theta(t)$ has a global maximum on $[0,\infty)$ at $t_0=\sqrt{\theta/(1-\theta)}$, with $g_\theta(t_0)= N_{\theta,\infty}^{-1}\leq 2^{-1/2}<1$, and is decreasing on $[t_0,\infty)$. Given $\epsilon>0$ set $f(t) = \epsilon^{-1} t\,\chi_{(t_0,t_0+\epsilon)} \,\phi$, for $t>0$, where $\chi_{(a,b)}$ denotes the characteristic function of $(a,b)\subset\R$. Then \eqref{eq:intcon} holds and
$$
\|\phi\|_{J_{\theta,1}(\overline X)} \leq \mathcal{J}_{\theta,1}(f)  = \frac{\|\phi\|_X}{\epsilon\, N_{\theta,\infty}} \int_{t_0}^{t_0+\epsilon} \frac{\rd t}{g_\theta(t)} \leq \frac{g_\theta(t_0)}{g_\theta(t_0+\epsilon)}\, \|\phi\|_X\, .
$$
As this holds for arbitrary $\epsilon >0$, $\|\phi\|_{J_{\theta,1}(\overline X)} \leq  \|\phi\|_X$.

On the other hand, if $\epsilon>0$ and $f$ satisfies \eqref{eq:intcon} with $\int_0^\infty (\|f(t)\|_X\,/t)\, \rd t < \infty$ and $\mathcal{J}_{\theta,1}(f)\leq \|\phi\|_{J_{\theta,1}(\overline X)} + \epsilon$, then, choosing $\eta\in (0,1)$ such that $\int_{(0,\infty)\setminus (\eta,\eta^{-1})} (\|f(t)\|_X/(tg_\theta(t)))\rd t <\epsilon$, it follows (since $g_\theta(t)\leq 1$) that $\|\phi-\phi_\eta\|_X < \epsilon$, where $\phi_\eta := \int_0^\infty (f_\eta(t)/t)\rd t$ and $f_\eta := f\, \chi_{(\eta,\eta^{-1})}$. Thus
$$
\|\phi\|_X -\epsilon \leq \|\phi_\eta\|_X = \lim_{q\to 1^+} \|\phi_\eta\|_{J_{\theta,q}(\overline X)}\leq \lim_{q\to 1^+} \mathcal{J}_{\theta,q}(f_\eta) =  \mathcal{J}_{\theta,1}(f_\eta) \leq  \mathcal{J}_{\theta,1}(f) + \epsilon \leq   \|\phi\|_{J_{\theta,1}(\overline X)} + 2\epsilon.
$$
As $\epsilon>0$ here is arbitrary, it follows that $\|\phi\|_{J_{\theta,1}(\overline X)} = \|\phi\|_X$.

\end{proof}

\section{Interpolation of Hilbert spaces} \label{sec:interpHil}

We focus in this section  on so-called {\em quadratic interpolation}, meaning the special case of interpolation where the compatible pairs are pairs of Hilbert spaces and the interpolation spaces are also Hilbert spaces. For the remainder of the paper we assume the normalisations \eqref{Ntheta} and \eqref{Jnormdef} for the $K$- and $J$-methods, and focus entirely on the case $q=2$, in which case the normalisation factors are given explicitly by \eqref{eq:Nchoice}. With the norms we have chosen, the $K$-method and $J$-method interpolation spaces $K_{\theta,2}(\overline X)$ and $J_{\theta,2}(\overline X)$ are Hilbert spaces (in fact, as we will see, the same Hilbert space if $\overline X$ is a Hilbert space compatible pair).

\subsection{The  \texorpdfstring{$K$}{K}- and \texorpdfstring{$J$}{J}-methods in the Hilbert space case} We begin with a result on real interpolation that at first sight appears to be a special case, but we will see later is generic.

\begin{thm} \label{thm:L2case}
Let $(\mathcal{X},M,\mu)$ be a measure space and let $\mathcal{Y}$ denote the set of measurable functions $\mathcal{X}\to \C$. Suppose that, for $j=0,1$, $w_j\in \mathcal{Y}$, with $w_j> 0$ almost everywhere, and let
$H_j:= L^2(\mathcal{X},M,w_j\mu)\subset \mathcal{Y}$, a Hilbert space with norm
$$
\|\phi\|_{H_j} := \left(\int_{\mathcal{X}} w_j |\phi|^2 \, \rd\mu\right)^{1/2}\,, \quad \mbox{for } \phi \in H_j \, .
$$
For $0<\theta < 1$, where $w_\theta:= w_0^{1-\theta}w_1^\theta$, let $H^\theta := L^2(\mathcal{X}, M, w_\theta\mu)$,  a Hilbert space with norm
$$
\|\phi\|_{H^\theta} := \left(\int_{\mathcal{X}} w_\theta|\phi|^2 \, \rd\mu\right)^{1/2},\quad \mbox{for } \phi \in H^\theta \, .
$$
Then, for $0<\theta<1$, where $\overline H=(H_0,H_1)$,
$$
H^\theta=K_{\theta,2}(\overline H) = J_{\theta,2}(\overline H),
$$
with equality of norms.
\end{thm}
\begin{proof}
We have to show that, for $\phi\in \Sigma(\overline H)$, $\|\phi\|_{H^\theta}=\|\phi\|_{K_{\theta,2}(\overline H)}=\|\phi\|_{J_{\theta,2}(\overline H)}$, for $0<\theta<1$. Now
\begin{eqnarray*}
\|\phi\|^2_{K_{\theta,2}(\overline H)}
& = & N^2_{\theta,2} \int_0^\infty t^{-1-2\theta} \inf_{\phi_0+\phi_1=\phi} \left(\|\phi_0\|^2_{H_0} + t^2 \|\phi_1\|^2_{H_1}\right) \,\rd t
\end{eqnarray*}
and
\begin{eqnarray*}
\|\phi_0\|^2_{H_0} + t^2 \|\phi_1\|^2_{H_1} = \int_{\mathcal{X}}\left( w_0 |\phi_0|^2 + w_1t^2 |\phi_1|^2\right)\, \rd \mu.
\end{eqnarray*}
Further (by applying \cite[B.4, p. 333]{McLean} pointwise inside the integral),
$$
\inf_{\phi_0+\phi_1=\phi}  \int_{\mathcal{X}}\left( w_0 |\phi_0|^2 + w_1t^2 | \phi_1|^2\right)\, \rd \mu = \int_{\mathcal{X}} \frac{w_0w_1t^2}{w_0+t^2w_1} | \phi|^2\, \rd \mu,
$$
this infimum achieved by $\phi_1 = w_0 \phi/(w_0+t^2w_1)$. Hence, and by Tonelli's theorem and \eqref{useful},
\begin{eqnarray*}
\|\phi\|^2_{K_{\theta,2}(\overline H)}
& = & N^2_{\theta,2} \int_{\mathcal{X}}\left(\int_0^\infty t^{-1-2\theta} \frac{w_0w_1t^2}{w_0+t^2w_1} \, \rd t\right) | \phi|^2\, \rd \mu = \|\phi\|^2_{H^\theta}\, .
\end{eqnarray*}

Also,
\begin{eqnarray*}
\|\phi\|_{J_{\theta,2}(\overline H)}^2
& = & N^{-2}_{\theta,2} \inf_f\int_0^\infty t^{-1-2\theta}  \left(\|f(t)\|^2_{H_0} + t^2 \|f(t)\|^2_{H_1}\right) \,\rd t
\end{eqnarray*}
and
\begin{eqnarray*}
\|f(t)\|^2_{H_0} + t^2 \|f(t)\|^2_{H_1} = \int_{\mathcal{X}} (w_0 + w_1t^2) |f(t)|^2\, \rd \mu,
\end{eqnarray*}
so that, by Tonelli's theorem,
\begin{equation} \label{eq:norm}
\|\phi\|_{J_{\theta,2}(\overline H)}^2
 =  N^{-2}_{\theta,2} \inf_f\int_{\mathcal{X}}\left(\int_0^\infty t^{-1-2\theta}   (w_0 + w_1t^2) |f(t)|^2\, \rd t\right) \,\rd \mu.
\end{equation}
Now we show below that this infimum is achieved for the choice
\begin{equation} \label{eq:choice}
f(t) = \frac{t^{2\theta}\phi}{(w_0+w_1t^2)\int_0^\infty s^{2\theta-1}/(w_0+w_1s^2)\, \rd s} = \frac{w_\theta N_{\theta,2}^2 t^{2\theta}\phi}{ w_0+w_1t^2}\, ;
\end{equation}
to get the second equality we use that, from \eqref{useful},
$$
\int_0^\infty \frac{s^{2\theta-1}}{w_0+w_1s^2}\, \rd s = \int_0^\infty \frac{s^{1-2\theta}}{w_0s^2+w_1}\, \rd s = \frac{w_{1-\theta}}{N_{\theta,2}^2 w_0 w_1}= \frac{1}{w_\theta N_{\theta,2}^2}.
$$
Substituting from \eqref{eq:choice} in \eqref{eq:norm} gives that
$$\|\phi\|_{J_{\theta,2}(\overline H)}^2
 =  N^{2}_{\theta,2} \int_{\mathcal{X}} w_\theta^2 |\phi|^2 \left(\int_0^\infty \frac{t^{-1+2\theta}}{w_0 + w_1t^2}\, \rd t\right) \,\rd \mu
= \int_{\mathcal{X}} w_\theta| \phi|^2\, \rd \mu = \|\phi\|_{H^\theta}^2.
$$

It remains to justify that the infimum is indeed attained by \eqref{eq:choice}. We note first that the definition of $f$ implies that $\int_0^\infty (f(t)/t)\rd t = \phi$, so that \eqref{eq:intcon} holds. Now suppose that $g$ is another eligible function such that \eqref{eq:intcon} holds, and let $\delta = g-f$. Then
$\int_0^\infty (\delta(t)/t)\, \rd t = 0$
and, using \eqref{eq:choice},
\begin{eqnarray*}
& & \int_{\mathcal{X}}\left(\int_0^\infty t^{-1-2\theta}   (w_0 + w_1t^2) |g(t)|^2\, \rd t\right) \rd \mu - \int_{\mathcal{X}}\left(\int_0^\infty t^{-1-2\theta}   (w_0 + w_1 t^2) |f(t)|^2\, \rd t\right) \rd \mu\\
 &\geq&2\Re \int_{\mathcal{X}}\left(\int_0^\infty t^{-1-2\theta}   (w_0 + w_1t^2) f(t) \bar \delta(t)\rd t\right) \rd \mu = 2N_{\theta,2}^2\Re \int_{\mathcal{X}}w_\theta \phi \left(\int_0^\infty \frac{ \bar\delta(t)}{t}\rd t\right) \rd \mu= 0.
\end{eqnarray*}
\end{proof}

The following is a straightforward corollary of the above theorem.
\begin{cor} \label{thm:choiceoftheta}
Let $\overline H = (H_0,H_1)$ be a compatible pair of Hilbert spaces, $(\mathcal{X},M,\mu)$ be a measure space and let $\mathcal{Y}$ denote the set of measurable functions $\mathcal{X}\to \C$. Suppose that there exists a linear map $\mathcal{A}:\Sigma(\overline H)\to \mathcal{Y}$ and, for $j=0,1$, functions $w_j\in \mathcal{Y}$, with $w_j>0$ almost everywhere, such that the mappings $\mathcal{A}:H_j\to L^2(\mathcal{X},M,w_j\mu)$ are unitary isomorphisms. For $0<\theta < 1$ define intermediate spaces $H^\theta$, with $\Delta(\overline H)\subset H^\theta \subset \Sigma(\overline H)$, by
$$
H^\theta := \bigg\{\phi\in \Sigma(\overline H): \|\phi\|_{H^\theta} := \bigg(\int_{\mathcal{X}} \big|w_\theta \mathcal{A}\phi\big|^2 \, \rd\mu\bigg)^{1/2} < \infty\bigg\},
$$
where $w_\theta:= w_0^{1-\theta}w_1^\theta$. Then, for $0<\theta<1$, $H^\theta=K_{\theta,2}(\overline H) = J_{\theta,2}(\overline H)$, with equality of norms.
\end{cor}

In the next theorem we justify our earlier statement that the situation described in Theorem \ref{thm:L2case} is generic, the point being that it follows from the spectral theorem for bounded normal operators that every Hilbert space compatible pair is unitarily equivalent to a compatible pair of weighted $L^2$-spaces. We should make clear that, while our method of proof that the $K$-method and $J$-method produce the same interpolation space, with equality of norms, appears to be new, this result (for the somewhat restricted separable Hilbert space case, with $\Delta(\overline H)$ dense in $H_0$ and $H_1$) is claimed recently in Ameur \cite[Example~4.1]{Ameur2004} (see also \cite[Section~7]{Ameur2014}),  though the choice of normalisation, details of the argument, and the definition of the $J$-method norm appear inaccurate in \cite{Ameur2004}.

In the following theorem and subsequently, for a Hilbert space $H$, $(\cdot,\cdot)_H$ denotes the inner product on $H$.
We note that $\Delta(\overline H)$ and $\Sigma(\overline H)$ are Hilbert spaces if we equip them with the equivalent norms defined by $\|\phi\|^\prime_{\Delta(\overline H)} := J(1,\phi,\overline H)$ and $\|\phi\|^\prime_{\Sigma(\overline H)} := K(1,\phi,\overline H)$, respectively.
In the next theorem we use standard results (e.g., \cite[Section VI, Theorem 2.23]{Kato}) on non-negative, closed symmetric forms and their associated self-adjoint operators.

\begin{thm} \label{thm:JequalK} Suppose that $\overline H=(H_0,H_1)$ is a compatible pair of Hilbert spaces. Then, for $0<\theta<1$,  $\|\phi\|_{K_{\theta,2}(\overline H)} = \|\phi\|_{J_{\theta,2}(\overline H)}$, for $\phi\in (H_0,H_1)_{\theta,2}$. Further, where $H^\circ_1$ denotes the closure in $H_1$ of $\Delta(\overline H)$, defining the unbounded, self-adjoint, injective operator  $T:H^\circ_1\to H^\circ_1$ by
$$
(T\phi,\psi)_{H_1} = (\phi,\psi)_{H_0}, \quad \phi,\psi\in \Delta(\overline H),
$$
and where $S$ is the unique non-negative square root of $T$, it holds that 
$$
\|\phi\|_{H_0} = \|S\phi\|_{H_1} \; \mbox{ and } \; \|\phi\|_{K_{\theta,2}(\overline H)} = \|\phi\|_{J_{\theta,2}(\overline H)} =\|S^{1-\theta}\phi\|_{H_1}, \quad \mbox{for } \phi \in \Delta(\overline H),
$$
so that $K_{\theta,2}(\overline H)$ is the closure of $\Delta(\overline H)$ in $\Sigma(\overline H)$ with respect to the norm defined by $\|\phi\|_\theta := \|S^{1-\theta}\phi\|_{H_1}$.
\end{thm}
\begin{proof} 
For $j=0,1$, define the non-negative bounded, injective operator $A_j:\Delta(\overline H)\to \Delta(\overline H)$ by the relation $(A_j\phi,\psi)_{\Delta(\overline H)} = (\phi,\psi)_{H_j}$, for $\phi,\psi\in \Delta(\overline H)$, where $(\cdot,\cdot)_{\Delta(\overline H)}$ denotes the inner product induced by the norm $\|\cdot\|^\prime_{\Delta(\overline H)}$.
 By the spectral theorem \cite[Corollary 4, p.~911]{DSII} there exists a measure space $(\mathcal{X},M,\mu)$, a bounded $\mu$-measurable function $w_0$, and a unitary isomorphism $U:\Delta(\overline H)\to L^2(\mathcal{X},M,\mu)$ such that
\begin{equation*} 
A_0\phi = U^{-1}w_0U\phi, \quad \mbox{for }\phi\in \Delta(\overline H),
\end{equation*}
and $w_0> 0$ $\mu$-almost everywhere since $A_0$ is non-negative and injective. Defining $w_1 := 1-w_0$ we see that $A_1\phi = U^{-1}w_1 U\phi$, for $\phi\in \Delta(\overline H)$, so that also $w_1>0$ $\mu$-almost everywhere.

For $\phi\in \Delta(\overline H)$,
$$
\|\phi\|_{H_j}^2 = (U^{-1}w_jU\phi,\phi)_{\Delta(\overline H)} = (w_jU\phi,U\phi)_{L^2(\mathcal{X},M,\mu)} = \|U\phi\|_{L^2(\mathcal{X},M,w_j\mu)}^2, \quad \mbox{for } j=0,1.
$$
Thus, where (similarly to $H_1^\circ$) $H_0^\circ$ denotes the closure of $\Delta(\overline H)$ in $H_0$,   $U$ extends to an isometry $U:H^\circ_j\to L^2(\mathcal{X},M,w_j\mu)$ for $j=0,1$. These extensions are unitary operators since their range contains $L^2(\mathcal{X},M,\mu)$, which is dense in $L^2(\mathcal{X},M,w_j\mu)$ for $j=0,1$. Where $\overline H^\circ:=(H_0^\circ,H_1^\circ)$, $U$ extends further to a  linear operator $U:\Sigma(\overline H^\circ)\to \mathcal Y$, the space of $\mu$-measurable functions defined on $\mathcal{X}$. Thus, applying Corollary \ref{thm:choiceoftheta} and noting part (v) of Theorem \ref{thm:interp}, we see that
$H^\theta=K_{\theta,2}(\overline H) = J_{\theta,2}(\overline H)$, with equality of norms, where
$$
H^\theta := \big\{\phi\in \Sigma(\overline H): \|\phi\|_{H^\theta} := \|U\phi\|_{L^2(\mathcal{X},M,w_\theta\mu)} < \infty\big\},
$$
and $w_\theta := w_0^{1-\theta}w_1^\theta$.
Moreover, for $\phi\in \Delta(\overline H)$, the unbounded operator $T:H_1^\circ\to H_1^\circ$ satisfies $T\phi = U^{-1}(w_0/w_1)U\phi$ so that
$\|S^{1-\theta}\phi\|_{H_1}^2 = (T^{1-\theta}\phi,\phi)_{H_1} = (A_1T^{1-\theta}\phi,\phi)_{\Delta(\overline H)} = (w_\theta U\phi,U\phi)_{L^2(\mathcal{X},M,\mu)}= \|\phi\|_{H^\theta}^2$, for $0<\theta<1$, and $\|S\phi\|_{H_1}^2 = (w_0 U\phi,U\phi)_{L^2(\mathcal{X},M,\mu)}= \|\phi\|_{H_0}^2$.
\end{proof}

Suppose that $\phi\in \Delta(\overline H)$. The above proof shows that
$$
\|\phi\|_{H_j} = \left(\int_{\mathcal{X}}w_j|U\phi|^2\, \rd \mu\right)^{1/2}, \;\; \mbox{for } j=0,1, \quad \mbox{and} \quad \|\phi\|_{K_{\theta,2}(\overline H)} = \left(\int_{\mathcal{X}}w_\theta|U\phi|^2\, \rd \mu\right)^{1/2}, \;\; \mbox{for } 0<\theta<1,
$$
with $U\phi\in L^2(\mathcal{X},M, \mu)$, $0<w_j\leq 1$ $\mu$-almost everywhere, for $j=0,1$, and $w_\theta := w_0^{1-\theta}w_1^\theta$. It follows from the dominated convergence theorem that $\|\phi\|_{K_{\theta,2}(\overline H)}$ depends continuously on $\theta$ and that
$$
\lim_{\theta\to 0^+}\|\phi\|_{K_{\theta,2}(\overline H)} = \|\phi\|_{H_0} \quad \mbox{and} \quad \lim_{\theta\to 1^-}\|\phi\|_{K_{\theta,2}(\overline H)} = \|\phi\|_{H_1}.
$$

In the special case, considered in \cite{LiMaI}, that $H_0$ is densely and continuously embedded in $H_1$, when $\Delta(\overline H) = H_0$ and $\Sigma(\overline H) = H_1$, the above theorem can be interpreted as stating that $(H_0,H_1)_{\theta,2}$ is the domain of the unbounded self-adjoint operator $S^{1-\theta}:H_1\to H_1$ (and $H_0$ the domain of $S$), this a standard characterisation of the $K$-method interpolation spaces in this special case, see, e.g., \cite[p.~99]{LiMaI} or \cite{Bramble}.
The following theorem (cf., \cite[Theorem B.2]{Bramble}), further illustrating the application of Corollary \ref{thm:choiceoftheta}, treats the special case when $H_1\subset H_0$, with a compact and dense embedding (which implies that both $H_0$ and $H_1$ are separable).

\begin{thm} \label{thm:compact} Suppose that $\overline H=(H_0,H_1)$ is a compatible pair of Hilbert spaces, with $H_1$ densely and compactly embedded in $H_0$.  Then the operator $T:H_1\to H_1$, defined  by
\begin{equation*} 
(T\phi,\psi)_{H_1} = (\phi,\psi)_{H_0}, \quad \phi,\psi\in H_1,
\end{equation*}
is compact, self-adjoint, and injective, and there exists  an orthogonal basis, $\{\phi_j:j\in \N\}$, for $H_1$, where each $\phi_j$ is  an eigenvector of $T$ with corresponding eigenvalue $\lambda_j$. Further, $\lambda_1\geq \lambda_2 \geq ... > 0$ and $\lambda_j\to 0$ as $j\to\infty$. Moreover, normalising this basis so that $\|\phi_j\|_{H_0}=1$ for each $j$, it holds for $0<\theta< 1$ that
\begin{equation*}
(H_0,H_1)_{\theta,2} = \Bigg\{\phi = \sum_{k=1}^\infty a_j \phi_j \in H_0: \|\phi\|_\theta^* := \bigg(\sum_{j=1}^\infty \lambda_j^{-\theta} |a_j|^2\bigg)^{1/2} < \infty\Bigg\}.
\end{equation*}
 Further, for $0<\theta<1$, $\|\phi\|_\theta^* =\|\phi\|_{K_{\theta,2}(\overline H)} = \|\phi\|_{J_{\theta,2}(\overline H)}$, for $\phi \in (H_0,H_1)_{\theta,2}$, and, for $j=0,1$, $\|\phi\|_j^* =\|\phi\|_{H_j}$, for $\phi \in H_j$.
\end{thm}
\begin{proof} Clearly $T$ is injective and self-adjoint, and we see easily (this a standard argument) that $T$ is compact. 
The existence of an orthogonal basis of $H_1$ consisting of eigenvectors of $T$, and the properties of the eigenvalues claimed, follow from standard results \cite[Theorem 2.36]{McLean}, and the positivity and injectivity of $T$. Further,
$(\phi_i,\phi_j)_{H_0} = (T\phi_i,\phi_j)_{H_1} = \lambda_i (\phi_i,\phi_j)_{H_1}$.
Thus, normalising by $\|\phi_j\|_{H_0}=1$, it holds that
$(\phi_i,\phi_j)_{H_0} = \delta_{ij}$, and $(\phi_i,\phi_j)_{H_1} = \lambda_i^{-1} \delta_{ij}$,
for $i,j\in \N$.
Since $H_1$ is dense  in $H_0$, $\{\phi_j\}$ is an orthonormal basis of $H_0$. Further, for $\phi\in H_1$ and $j\in \N$,
$(\phi_j,\phi)_{H_1} = \lambda_j^{-1} (T\phi_j,\phi)_{H_1} = \lambda_j^{-1} (\phi_j,\phi)_{H_0}.$
 Thus, for $\phi\in H_0$,
$\|\phi\|_{H_0}^2 = \sum_{j=1}^\infty |(\phi_j,\phi)_{H_0}|^2$,
while, for $\phi\in H_1$,
$$
\|\phi\|_{H_1}^2 = \sum_{j=1}^\infty \lambda_i |(\phi,\phi_j)_{H_1}|^2  = \sum_{j=1}^\infty \lambda_i^{-1} |(\phi,\phi_j)_{H_0}|^2.
$$
To complete the proof we use Corollary \ref{thm:choiceoftheta}, with the measure space $(\N, 2^\N, \mu)$, where $\mu$ is counting measure, and where $\mathcal{A}\phi$, for $\phi\in H_0$, is the function $\mathcal{A}\phi:\N\to\C$ defined by $\mathcal{A}\phi(j) = (\phi,\phi_j)_{H_0}$, $j\in\N$, and $w_0$ and $w_j$ are defined by $w_0(j) = 1$ and $w_1(j) = \lambda_j^{-1/2}$, $j\in\N$.
\end{proof}

\subsection{Uniqueness of interpolation in the Hilbert space case} \label{sec:unique}

Theorem \ref{thm:JequalK} is a statement that, in the Hilbert space case, three standard methods of interpolation produce the same interpolation space, with the same norm.
This is illustrative of a more general result. It turns out, roughly speaking, that all methods of interpolation between Hilbert spaces that produce, for $0<\theta<1$, interpolation spaces that are Hilbert spaces and that are exact of exponent $\theta$, must coincide. To make a precise statement we need the following definition: given a Hilbert space compatible pair $\overline H = (H_0,H_1)$,  an intermediate space $H$ between $H_0$ and $H_1$ is said to be a {\em geometric interpolation space of exponent $\theta$} \cite{McCarthy}, for some $0<\theta<1$, relative to $\overline H$, if $H$ is a Hilbert space, $\Delta(\overline H)$ is dense in $H$, and the following three conditions hold for linear operators $T$:
\begin{enumerate}
\item[i)] If $T$ maps $\Delta(\overline H)$ to $\Delta(\overline H)$ and $\|T\phi\|_{H_0} \leq \lambda_0 \|\phi\|_{H_0}$ and $\|T\phi\|_{H_1} \leq \lambda_1\|\phi\|_{H_1}$, for all $\phi\in \Delta(\overline H)$, then $\|T\phi\|_{H} \leq \lambda_0^{1-\theta}\lambda_1^\theta \|\phi\|_{H}$, for all $\phi\in \Delta(\overline H)$;
\item[ii)] If $T$ maps $\Delta(\overline H)$ to $\mathcal{H}$, for some Hilbert space $\mathcal{H}$, and $\|T\phi\|_{\mathcal{H}} \leq \lambda_0 \|\phi\|_{H_0}$ and $\|T\phi\|_{\mathcal{H}} \leq \lambda_1\|\phi\|_{H_1}$, for all $\phi\in \Delta(\overline H)$, then $\|T\phi\|_{\mathcal{H}} \leq \lambda_0^{1-\theta}\lambda_1^\theta \|\phi\|_{H}$, for all $\phi\in \Delta(\overline H)$;
\item[iii)] If $T$ maps $\mathcal{H}$ to $\Delta(\overline H)$, for some Hilbert space $\mathcal{H}$, and $\|T\phi\|_{H_0} \leq \lambda_0 \|\phi\|_{\mathcal{H}}$ and $\|T\phi\|_{H_1} \leq \lambda_1\|\phi\|_{\mathcal{H}}$, for all $\phi\in \mathcal{H}$, then $\|T\phi\|_{H} \leq \lambda_0^{1-\theta}\lambda_1^\theta \|\phi\|_{\mathcal{H}}$, for all $\phi\in \mathcal{H}$.
\end{enumerate}
More briefly but equivalently, in the language introduced in Section~\ref{sec:interp}, $H$ is a geometric interpolation space of exponent $\theta$ if $\Delta(\overline H)\subset H\subset \Sigma(\overline H)$, with continuous embeddings and $\Delta(\overline H)$ dense in $H$, and if: (i) $(H,H)$ is a pair of interpolating spaces relative to $(\overline H,\overline H)$ that is exact of exponent $\theta$; and (ii) for every Hilbert space $\mathcal{H}$, where $\overline{\mathcal{H}} := (\mathcal{H},\mathcal{H})$, $(H,\mathcal{H})$ and $(\mathcal{H},H)$ are pairs of interpolation spaces, relative to $(\overline H,\overline{\mathcal{H}})$ and $(\overline{\mathcal{H}},\overline H)$, respectively, that are exact of exponent $\theta$.

The following is the key uniqueness and existence theorem; the uniqueness part is due to McCarthy \cite{McCarthy} in the separable Hilbert space case with $\Delta(\overline H)$ dense in $H_0$ and $H_1$.
We emphasise that this theorem states that, given a Hilbert space compatible pair, two geometric interpolation spaces with the same exponent must have equal norms, not only equivalent norms.

\begin{thm} \label{thm:unique}Suppose that $\overline H = (H_0,H_1)$ is a compatible pair of Hilbert spaces. Then, for $0<\theta<1$, $K_{\theta,2}(\overline H)$ is the unique geometric interpolation space of exponent $\theta$ relative to $\overline H$.
\end{thm}
\begin{proof} That $H_\theta:= K_{\theta,2}(\overline H)$ is a geometric interpolation space of exponent $\theta$ follows from Lemma \ref{intnormub}(iii) and Theorem \ref{thm:interp} (i) and (iv).
To see that $H_\theta$ is the unique geometric interpolation space we adapt the argument of \cite{McCarthy}, but using the technology (and notation) of the proof of Theorem \ref{thm:JequalK}.

So suppose that $G$ is another geometric interpolation space of exponent $\theta$ relative to $\overline H$. To show that $G=H_\theta$ with equality of norms it is enough to show that $\|\phi\|_G = \|\phi\|_{H_\theta}$, for $\phi\in \Delta(\overline H)$.

Using the notation of the proof of Theorem \ref{thm:JequalK}, recall that $T:H^\circ_1\to H^\circ_1$ is given by $T=U^{-1}\omega U$, where $\omega:= w_0/w_1$. For $0\leq a<b$ let $\chi_{a,b}\in \mathcal{Y}$ denote the characteristic function of the set $\{x\in \mathcal{X}: a\leq \omega(x) < b\}$, and define the projection operator  $P(a,b):\Sigma(\overline H^\circ)\to \Delta(\overline H)$ by $P(a,b)\phi = U^{-1} \chi_{a,b} U\phi$. Recalling that $U:H_j^\circ\to L^2(\mathcal{X},M,w_j\mu)$ is unitary, we see that the mapping $P(a,b):H_j^\circ \to H_j^\circ$ has norm one, for $j=0,1$: since $G$ and $H_\theta$ are geometric interpolation spaces, also $P(a,b):G\to G$ and $P(a,b):H_\theta\to H_\theta$ have norm one. Thus $P(a,b)$ is an orthogonal projection operator on each of $H_j^\circ$, $j=0,1$, $G$, and $H_\theta$, for otherwise there exists a $\psi$ in the null-space of $P(a,b)$ which is not orthogonal to some $\phi$ in the range of $P(a,b)$, and then, for some $\eta\in\C$, $\|\phi\|> \|\phi+\eta \psi\| \geq
\|P(\phi+\eta\psi)\| =\|\phi\|$, a contradiction.

Let $\mathcal{H}$ denote the range of $P(a,b):\Sigma(\overline H)\to \Delta(\overline H)$ equipped with the norm of $H_1$. Clearly $P(a,b):H^\circ_1\to \mathcal{H}$ has norm one, while it is a straightforward calculation that $P(a,b):H^\circ_0\to \mathcal{H}$ has norm $\leq \|\chi_{a,b} \omega^{-1/2}\|_{L^\infty(\mathcal{X},M,\mu)}\leq a^{-1/2}$, so that $P(a,b):G\to \mathcal{H}$ has norm $\leq a^{-(1-\theta)/2}$. Similarly, where $R$ is the inclusion map (so $R\phi=\phi$), $R:\mathcal{H}\to H^\circ_1$ has norm one, $R:\mathcal{H}\to H^\circ_0$ has norm $\leq \|\chi_{a,b} \omega^{1/2}\|_{L^\infty(\mathcal{X},M,\mu)}\leq b^{1/2}$, so that $R:\mathcal{H}\to G$ has norm $\leq b^{(1-\theta)/2}$. Thus, for $\phi \in \mathcal{H}$,
\begin{equation} \label{MC2}
a^{1-\theta}\|\phi\|^2_{H_1} \leq \|\phi\|^2_G\leq b^{1-\theta}\|\phi\|^2_{H_1}.
\end{equation}

Finally, for every $p>1$, we observe that, for $\phi \in \Delta(\overline H)$, where $\phi_n := P(p^n,p^{n+1})\phi$, since $\{x:\omega(x)=0\}$ has $\mu$-measure zero,
\begin{equation} \label{MC1}
\|\phi\|_{H_\theta}^2 = \sum_{n=-\infty}^\infty \|\phi_n\|^2_{H_\theta}
\quad\mbox{ and } \quad
\|\phi\|_{G}^2 = \sum_{n=-\infty}^\infty \|\phi_n\|^2_{G}\, .
\end{equation}
Further, for each $n$,
$$
\|\phi_n\|_{H_\theta}^2 = \int_{\mathcal X} w_\theta |U\phi_n|^2 \, \rd \mu = \int_{\mathcal X} \chi_{p^n,p^{n+1}}w_\theta |U\phi_n|^2 \, \rd \mu = \int_{\mathcal X} \chi_{p^n,p^{n+1}}\omega^{1-\theta}w_1 |U\phi_n|^2 \, \rd \mu
$$
so that
$$
p^{n(1-\theta)} \|\phi_n\|_{H_1}^2\leq \|\phi_n\|_{H_\theta}^2 \leq p^{(n+1)(1-\theta)} \|\phi_n\|_{H_1}^2\, .
$$
Combining these inequalities with \eqref{MC2} (taking $a=p^n$, $b=p^{n+1}$) and \eqref{MC1}, we see that
$$
p^{-(1-\theta)} \|\phi\|_{G}^2\leq \|\phi\|_{H_\theta}^2 \leq p^{1-\theta} \|\phi\|_{G}^2.
$$
Since this holds for all $p>1$, $\|\phi\|_{H_\theta}=\|\phi\|_G$.
\end{proof}

\begin{rem} \label{rem:functor} For those who like the language of category theory (commonly used in the interpolation space context, e.g.~\cite{BeLo,Triebel78ITFSDO}), the above theorem says that there exists a unique functor $F$ from the category
of Hilbert space compatible pairs to the category
of Hilbert spaces such that:
\emph{(i)} for every Hilbert space compatible pair $\overline H$,  $\Delta(\overline H) \subset F(\overline H) \subset \Sigma (\overline H)$, with the embeddings continuous and $\Delta(\overline H)$ dense in $F(\overline H)$;
\emph{(ii)} for every Hilbert space $H$, where $\overline H := (H,H)$, it holds that $F(\overline H) = H$;
\emph{(iii)} for every pair $(\overline H,\overline{\mathcal{H}})$ of Hilbert space compatible pairs,  $(F(\overline H),F(\overline{\mathcal{H}}))$ is a pair of interpolation spaces, relative to $(\overline H,\overline{\mathcal{H}})$, that is exact of exponent $\theta$. Theorems \ref{thm:unique}, \ref{thm:interp} (i), and \ref{thm:JequalK} tell us that the $K$-method and the $J$-method are both instances of this functor. It follows from Theorems 4.1.2, 4.2.1(c), 4.2.2(a) in \cite{BeLo} that the complex interpolation method is also an instance of this functor, so that, for every Hilbert space compatible pair $\overline H=(H_0,H_1)$, the standard complex interpolation space $(H_0,H_1)_{[\theta]}$ (in the
notation of \cite{BeLo}) coincides with $K_{\theta,2}(\overline H)$, with equality of norms.
\end{rem}

\subsection{Duality and interpolation scales} \label{sec:dualscales}

Theorems \ref{thm:JequalK} and \ref{thm:unique} and Remark \ref{rem:functor} make clear that life is simpler in the Hilbert space case. Two further simplications are captured in the following theorem (cf., Theorem \ref{thm:interp} (vii) and Theorem \ref{JKdual}).

\begin{thm} \label{simpler} Suppose  that $\overline H=(H_0, H_1)$ is a Hilbert space compatible pair.
\begin{enumerate}
\item[(i)] If $\theta_0,\theta_1\in [0,1]$, and, for $j=0,1$, $\mathcal{H}_j := (H_0,H_1)_{\theta_j,2}$, if $0<\theta_j<1$, while $\mathcal{H}_j:= H_{\theta_j}$, if $\theta_j\in \{0,1\}$, then $(\mathcal{H}_0,\mathcal{H}_1)_{\eta,2} = (H_0,H_1)_{\theta,2}$, with equal norms,
 for  $\theta = (1-\eta)\theta_0 + \eta \theta_1$ and $0<\eta<1$.
\item[(ii)] If
 $\Delta(\overline H)$ is dense in $H_0$ and $H_1$, so that $\overline H^*:=(H_0^*,H_1^*)$ is a Hilbert space compatible pair, then
$$
(H_0,H_1)_{\theta,2}^* = (H_0^*,H_1^*)_{\theta,2},
$$
for $0<\theta<1$, with equality of norms.
\end{enumerate}
\end{thm}
\begin{proof}
To prove (i) we note first that, by Theorem \ref{thm:interp} (v), we can assume $\Delta(\overline H)$ is dense in $H_0$ and $H_1$.

With this assumption it holds---see the proof of Theorem \ref{thm:JequalK}---that there exists a measure space $(\mathcal{X},M,\mu)$, a unitary operator $U:\Delta(\overline H)\to L^2(\mathcal{X},M,\mu)$, and functions $w_j:\mathcal{X}\to [0,\infty)$ that are $\mu$-measurable and non-zero $\mu$-almost everywhere, such that $U:H_j\to L^2(\mathcal{X},M,w_j\mu)$ is a unitary operator for $j=0,1$ and $U:(H_0,H_1)_{\theta,2}\to L^2(\mathcal{X},M,w_\theta \mu)$ is a unitary operator for $0<\theta<1$, where $w_\theta:=w_0^{1-\theta}w_1^\theta$. But this identical argument, repeated for $(\mathcal{H}_0,\mathcal{H}_1)_{\eta,2}$, gives that $U:(\mathcal{H}_0,\mathcal{H}_1)_{\eta,2}\to L^2(\mathcal{X},M,W_\eta\mu)$ is a unitary operator, where $W_\eta := W_0^{1-\eta}W_1^\eta$ and $W_j := w_{\theta_j}$. But this proves the result since $W_\eta=w_\theta$ if $\theta = (1-\eta)\theta_0 + \eta \theta_1$ and $0<\eta<1$.

That (ii) holds is immediate from Theorems \ref{JKdual} and \ref{thm:JequalK}.
\end{proof}

\begin{rem} \label{rem:scales} A useful concept, used extensively in Section~\ref{sec:sob} below, is that of an interpolation scale. Given a closed interval $\mathcal{I}\subset \R$ (e.g., $\mathcal{I} = [a,b]$, for some $a<b$, $\mathcal{I} = [0,\infty)$, $\mathcal{I} = \R$) we will say that a collection of Hilbert spaces $\{H_s:s \in \mathcal{I}\}$, indexed by $\mathcal{I}$, is an {\em interpolation scale} if, for all $s,t\in \mathcal{I}$ and $0<\eta<1$,
$$
(H_s,H_t)_{\eta,2} = H_\theta, \quad \mbox{for }  \theta = (1-\eta)s + \eta t.
$$
We will say that $\{H_s:s \in \mathcal{I}\}$  is an {\em exact} interpolation scale if, moreover, the norms of $(H_s,H_t)_{\eta,2}$ and $H_\theta$ are equal, for $s,t\in\mathcal{I}$ and $0<\eta<1$.

In this terminology part (i) of the above theorem is precisely a statement that, for every Hilbert space compatible pair $\overline H=(H_0,H_1)$, where $H_s:= (H_0,H_1)_{s,2}$, for $0<s<1$, $\{H_s:0\leq s\leq 1\}$ is an exact interpolation scale. If $\Delta(\overline H)$ is dense in $H_0$ and $H_1$, part (ii) implies that also $\{H^*_s:0\leq s\leq 1\}$ is an exact interpolation scale.
\end{rem}

\section{Interpolation of Sobolev spaces} \label{sec:sob}

In this section we study Hilbert space interpolation, analysed in Section~\ref{sec:interpHil}, applied to the classical Sobolev spaces $H^s(\Omega)$ and $\tH^s(\Omega)$, for $s\in \R$ and an open set $\Omega$. (Our notations here, which we make precise below, are those of \cite{McLean}.) This is a classical topic of study (see, e.g., notably \cite{LiMaI}). Our results below provide a more complete answer than hitherto available to the following questions:
\begin{enumerate}
\item[(i)] Let $H_s$, for $s\in \R$, denote $H^s(\Omega)$ or $\tH^s(\Omega)$. For which classes of $\Omega$ and what range of $s$ is $\{H_s\}$ an (exact) interpolation scale?
\item[(ii)] In cases where $\{H_s\}$ is an interpolation scale but not an exact interpolation scale, how different are the $H_s$ norm and the interpolation norm?
\end{enumerate}
Our answers to (i) and (ii) will consist mainly of examples and counterexamples. In particular, in the course of answering these questions we will write down, in certain cases of interest, explicit expressions for interpolation norms that may be of some independent interest.
Our investigations in this section are in very large part prompted and inspired by the results and discussion in \cite[Appendix B]{McLean}, though we will exhibit a counterexample to one of the results claimed in \cite{McLean}.

We talk a little vaguely in the above paragraph about ``Hilbert space interpolation''.
This vagueness is justified in Section~\ref{sec:unique} which makes clear that, for $0<\theta<1$, there is only one method of interpolation of a pair of compatible Hilbert spaces $\overline H=(H_0,H_1)$ which produces an interpolation space $\overline H_\theta$ that is a geometric interpolation space of exponent $\theta$ (in the terminology of \S\ref{sec:unique}).
Concretely this intermediate space is given both by the real interpolation methods, the $K$- and $J$-methods with $q=2$, and by the complex interpolation method: to emphasise, these methods give the identical interpolation space with identical norm (with the choice of normalisations we have made for the $K$- and $J$-methods).
We will, throughout this section, use  $\overline H_\theta$ and $(H_0,H_1)_\theta$ as our notations for this interpolation space and $\|\cdot\|_{\overline H_\theta}$ as our notation for the norm, so that $\overline H_\theta = (H_0,H_1)_\theta$ and $\|\cdot\|_{\overline H_\theta}$ are abbreviations for $(H_0,H_1)_{\theta,2} = K_{\theta,2}(\overline H)= J_{\theta,2}(\overline H)$ and
$\|\cdot\|_{K_{\theta,2}(\overline H)}=\|\cdot\|_{J_{\theta,2}(\overline H)}$, respectively, the latter defined with the normalisation \eqref{eq:Nchoice}.

\subsection{The spaces  \texorpdfstring{$H^s(\R^n)$}{Hs(Rn)}}

Our function space notations and definitions will be  those in \cite{McLean}. For $n\in\N$ let $\scrS(\R^n)$ denote the Schwartz space of smooth rapidly decreasing functions, and $\scrS^*(\R^n)$ its dual space, the space of tempered distributions. For $u\in \scrS(\R^n)$, $v\in \scrS^*(\R^n)$, we denote by $\langle u,v\rangle$ the action of $v$ on $u$, and we embed $L^2(\R^n)\supset \scrS(\R^n)$ in $\scrS^*(\R^n)$ in the usual way, i.e., $\langle u,v\rangle:= (u,v)$, where $(\cdot,\cdot)$ denotes the usual inner product on $L^2(\R^n)$, in the case that $u\in \scrS(\R^n)$, $v\in L^2(\R^n)$.  We define the Fourier transform $\hat{u}={\cal F} u\in \scrS(\R^n)$, for $u\in \scrS(\R^n)$, by
\begin{align*}
\hat{u}(\xi):= \frac{1}{(2\pi)^{n/2}}\int_{\R^n}\re^{-\ri \xi\cdot x}u(x)\,\rd x ,
\quad \mbox{for }\xi\in\R^n, 
\end{align*}
and then extend $\mathcal{F}$ to a mapping from $\scrS^*(\R^n)$ to $\scrS^*(\R^n)$ in the canonical way, e.g., \cite{McLean}. For $s\in \R$
we define the Sobolev space $H^s(\R^n)\subset \scrS^*(\R^n)$ by
\begin{equation} \label{eq:Hsdef}
H^s(\R^n):=\bigg\{u\in \scrS^*(\R^n): \|u\|_{H^s(\R^n)} :=
\bigg(\int_{\R^n}(1+|\xi|^2)^{s}|\hat{u}(\xi)|^2\,\rd \xi\bigg)^{1/2}<\infty\bigg\}.
\end{equation}
$\scrD(\R^n)\subset \scrS(\R^n)\subset H^s(\R^n)$ are dense in $H^s(\R^n)$ \cite{McLean} (for an open set $\Omega$, $\scrD(\Omega)$ denotes the space of those $u\in C^\infty(\Omega)$ that are compactly supported in $\Omega$). Hence and from \eqref{eq:Hsdef} it is clear that $H^t(\R^n)$ is continuously and densely embedded in $H^s(\R^n)$, for $s<t$, with $\|u\|_{H^s(\R^n)}\leq \|u\|_{H^t(\R^n)}$, for $u\in H^t(\R^n)$. By Plancherel's theorem, $H^0(\R^n)=L^2(\R^n)$ with equality of norms, so that $H^s(\R^n)\subset L^2(\R^n)$ for $s\geq 0$. Moreover, from the definition \eqref{eq:Hsdef},
\begin{equation} \label{eq:normint}
\|u\|^2_{H^{m}(\R^n)} = \sum_{|\alpha|\leq m}
\binom{m}{|\alpha|}\binom{|\alpha|}{\alpha} 
\|\partial^\alpha u\|_{L^2(\R^n)}^2, \quad \mbox{for } m\in \N_0:=\N\cup \{0\},
\end{equation}
where, for $\alpha = (\alpha_1,...\alpha_n)\in \N_0^n$,
$|\alpha| := \sum_{i=1}^n \alpha_i$,
$\binom{|\alpha|}{\alpha}:=|\alpha|!/(\alpha_1!\cdots\alpha_n!)$,
$\partial^\alpha:= \prod_{i=1}^n\partial^{\alpha_i}_i$, and $\partial_j := \partial /\partial x_j$ (the derivative in a distributional sense).

The following result (\!\cite[Theorem B.7]{McLean}) is all there is to say about $H^s(\R^n)$ and Hilbert space interpolation.
\begin{thm} \label{thm:wholespace} $\{H^s(\R^n):s\in \R\}$ is an exact interpolation scale, i.e., for $s_0,s_1\in \R$, $0<\theta<1$, $(H^{s_0}(\R^n),H^{s_1}(\R^n))_\theta = H^s(\R^n)$, with equality of norms, if $s = s_0(1-\theta)+s_1\theta$.
\end{thm}
\begin{proof} This follows from Corollary \ref{thm:choiceoftheta}, applied with $\mathcal{A}=\mathcal{F}$, $\mathcal X=\R^n$, and $w_j(\xi) = (1+|\xi|^2)^{s_j/2}$. 
\end{proof}


\subsection{The spaces  \texorpdfstring{$H^s(\Omega)$}{Hs(Omega)}}

For $\Omega\subset \R^n$ there are at least two definitions of $H^s(\Omega)$ in use (equivalent if $\Omega$ is sufficiently regular). Following \cite{McLean} (or see \cite[Section~4.2.1]{Triebel78ITFSDO}), we will define
$$
H^s(\Omega):= \big\{u\in \scrD^*(\Omega): u = U|_\Omega, \mbox{ for some } U\in H^s(\R^n)\big\},
$$
where $\scrD^*(\Omega)$ denotes the space of Schwartz distributions, the continuous linear functionals on $\scrD(\Omega)$ \cite[p.~65]{McLean}, and $U|_\Omega$ denotes the restriction of $U\in \scrD^*(\R^n)\supset \scrS^*(\R^n)$ to $\Omega$.  $H^s(\Omega)$ is endowed with the norm
$$
\|u\|_{H^s(\Omega)} := \inf \left\{\|U\|_{H^s(\R^n)}:U|_\Omega = u\right\}, \quad \mbox{for } u\in H^s(\Omega).
$$
With this norm, for $s\in \R$, $H^s(\Omega)$ is a Hilbert space, $\scrD(\overline \Omega):= \{U|_\Omega:U\in \scrD(\R^n)\}$ is dense in $H^s(\Omega)$, and  $H^t(\Omega)$ is continuously and densely embedded in $H^s(\Omega)$ with $\|u\|_{H^s(\Omega)}\leq \|u\|_{H^t(\Omega)}$,  for $s<t$ and $u\in H^t(\Omega)$ \cite{McLean}. Further $L^2(\Omega)=H^0(\Omega)$ with equality of norms, so that $H^s(\Omega)\subset L^2(\Omega)$ for $s>0$.

For $m\in\N_0$, let
$$
W^m_2(\Omega):=\big\{u\in L^2(\Omega):\partial^\alpha u \in L^2(\Omega)\mbox{ for } 
|\alpha|\leq m\big\}
$$
(our notation is that of \cite{McLean} and other authors, but the notation $H^m(\Omega)$ for this space is very common, e.g.~\cite{LiMaI}). $W^m_2(\Omega)$ is a Hilbert space with the norm
$$
\|u\|_{W_2^m(\Omega)} := \Bigg(\sum_{|\alpha|\leq m} a_{\alpha,m} \|\partial^\alpha u\|_{L^2(\Omega)}^2\Bigg)^{1/2},
$$
for any choice of positive coefficients $a_{\alpha,m}$. The usual choice is $a_{\alpha,m}=1$, but, comparing with \eqref{eq:normint}, we see that the choice
$
a_{\alpha,m} = \binom{m}{|\alpha|}\binom{|\alpha|}{\alpha}
$
 ensures that
$
\|u\|_{H^m(\Omega)} \geq \|u\|_{W_2^m(\Omega)}$, for $u\in H^m(\Omega),
$
 with equality when $\Omega = \R^n$. Clearly, $H^m(\Omega)$ is continuously embedded in $W^m_2(\Omega)$, for all $\Omega$.

Whether $H^s(\Omega)$ is an interpolation scale depends on the smoothness of $\Omega$. As usual (see, e.g., \cite[p.~90]{McLean}),  for $m\in \N_0$ we will say that $\Omega\subset \R^n$ is a $C^{m}$ open set if $\partial \Omega$ is bounded and if, in a neighbourhood of every point on $\partial \Omega$, $\partial \Omega$ is (in some rotated coordinate system) the graph of a $C^m$ function, and the part of $\Omega$ in that neighbourhood is part of the hypograph of the same function. Likewise, for $0<\beta\leq 1$, we will say that $\Omega$ is a $C^{0,\beta}$ open set, if it is a $C^0$ open set and $\partial \Omega$ is locally the graph of a function that is H\"older continuous of index $\beta$. In particular, a $C^{0,1}$ or {\em Lipschitz} open set has boundary that is locally the graph of a Lipschitz continuous function.
We say that $\{x\in\R^n:x_n<\ell(x_1,\ldots,x_{n-1})\}$ is a \emph{Lipschitz hypograph} if $\ell:\R^{n-1}\to\R$ is a Lipschitz function.

Let $\mathcal{R}:H^s(\R^n)\to H^s(\Omega)$ be the restriction operator, i.e., $\mathcal{R}U=U|_\Omega$, for $U\in H^s(\R^n)$: this is an operator with norm one for all $s\in \R$. It is clear that $W^m_2(\Omega) = H^m(\Omega)$, with equivalence of norms, if there exists a continuous {\em extension operator} $\mathcal{E}:W^m_2(\Omega)\to H^m(\R^n)$ that is a right inverse to $\mathcal{R}$, so that $\mathcal{RE}u = u$, for all $u\in W^m_2(\Omega)$. Such an extension operator is also a continuous extension operator $\mathcal{E}: H^m(\Omega) \to H^m(\R^n)$. An extension operator $\mathcal{E}_s: H^s(\Omega) \to H^s(\R^n)$ exists for all $\Omega$ and all $s\in \R$: for $u\in H^s(\Omega)$, set $U:= \mathcal{E}_s u$ to be the unique minimiser in $H^s(\R^n)$ of $\|U\|_{H^s(\R^n)}$ subject to $U|_\Omega = u$ (see \cite[p.~77]{McLean}). These operators $\mathcal{E}_s$ have norm one for all $s\in \R$ and all $\Omega$. But whether $H^s(\Omega)$ is an interpolation scale, for some range of $s$, depends on the existence of
an extension operator which, {\em simultaneously}, maps $H^{s}(\Omega)$ to $H^s(\R^n)$, for two distinct values of $s$. The following lemma is a quantitative version of standard arguments, e.g.~\cite[Section 4.3]{Triebel78ITFSDO}.

\begin{lem} \label{lem:ext} Suppose that $s_0\leq s_1$, $0<\theta<1$, and set $s = s_0(1-\theta) + s_1\theta$, $\overline H =(H^{s_0}(\Omega),H^{s_1}(\Omega))$. Then $H^s(\Omega)\subset \overline H_\theta = (H^{s_0}(\Omega),H^{s_1}(\Omega))_\theta$, with $\|u\|_{\overline H_\theta}\leq \|u\|_{H^s(\Omega)}$, for $u\in H^s(\Omega)$. If also, for some $\lambda_0,\lambda_1\geq 1$,
 $\mathcal{E}:H^{s_0}(\Omega)\to H^{s_0}(\R^n)$ is an extension operator, with $\|\mathcal{E}u\|_{H^{s_j}(\R^n)}\leq \lambda_j \|u\|_{H^{s_j}(\Omega)}$ for $u\in H^{s_1}(\Omega)$ and $j=0,1$, then $H^s(\Omega)=\overline H_\theta$ with equivalence of norms, precisely with
 $$
 \lambda_0^{\theta-1}\lambda_1^{-\theta}\|u\|_{H^s(\Omega)}\leq \|u\|_{\overline H_\theta} \leq   \|u\|_{H^s(\Omega)}, \quad \mbox{for } u\in H^s(\Omega).
 $$
 Further, $\{H^s(\Omega): s_0\leq s\leq s_1\}$ is an interpolation scale.
\end{lem}
\begin{proof} By Theorem \ref{thm:wholespace}, $(H^{s_0}(\R^n),H^{s_1}(\R^n))_\theta = H^s(\R^n)$ with equality of norms. For all $t\in \R$, $\mathcal{R}:H^t(\R^n)\to H^t(\Omega)$ has norm one. Thus, by Theorem \ref{thm:interp} (i), $H^s(\Omega) = \mathcal{R}(H^s(\R^n)) \subset \overline H_\theta$ and $\mathcal{R}:H^s(\R^n) \to \overline H_\theta$ with norm one, so that, for $u\in H^s(\Omega)$, $\|u\|_{\overline H_\theta} = \|\mathcal{R}\mathcal{E}_s u\|_{\overline H_\theta} \leq \|\mathcal{E}_s u\|_{H^s(\R^n)}= \|u\|_{H^s(\Omega)}$, where $\mathcal{E}_su$ is the extension with minimal $H^s(\R^n)$ norm, described above.

If also the extension operator $\mathcal{E}$ has the properties claimed, then, by Theorem \ref{thm:interp} (i), $\mathcal{E}(\overline H_\theta) \subset H^s(\R^n)$ so that $\overline H_\theta = \mathcal{R}\mathcal{E}(\overline H_\theta) \subset \mathcal{R}(H^s(\R^n))=H^s(\Omega)$. Further, $\mathcal{E}:\overline H_\theta \to H^s(\R^n)$ with norm $\leq \lambda_0^{1-\theta}\lambda_1^\theta$, so that, for $u\in H^s(\Omega)=\overline H_\theta$, $\|u\|_{H^s(\Omega)}=\|\mathcal{R}\mathcal{E} u\|_{H^s(\Omega)} \leq \|\mathcal{E}u\|_{H^s(\R^n)} \leq \lambda_0^{1-\theta}\lambda_1^\theta \|u\|_{\overline H_\theta}$.

Hence, noting Theorem \ref{simpler} (i), $\{H^s(\Omega): s_0\leq s\leq s_1\}$ is an interpolation scale.
\end{proof}

\begin{example} \label{ex:hs} As an example, consider the case that $\Omega$ is the half-space $\Omega = \{x=(x_1,...,x_n):x_1>0\}$, $s_0=0$, and $s_1=1$. In this case a simple extension operator is just reflection: $\mathcal{E}u(x) := u(x)$, for $x_1\geq 0$, and $\mathcal{E}u(x) := u(x^\prime)$, for $x_1<0$, where $x^\prime = (-x_1, x_2,...,x_n)$. In this example $\mathcal{E}:H^s(\Omega)\to H^s(\R^n)$ has norm $2$ for $s=0,1$ and, applying Lemma \ref{lem:ext}, $H^s(\Omega) = \overline H_s := (L^2(\Omega),H^1(\Omega))_s$, for $0<s<1$, with
$$
\frac{1}{2}\|u\|_{H^s(\Omega)}\leq \|u\|_{\overline H_s} \leq   \|u\|_{H^s(\Omega)}, \quad \mbox{for } u\in H^s(\Omega).
 $$
\end{example}

The construction of a continuous extension operator $\mathcal{E}:W^m_2(\Omega)\to H^m(\R^n)$, for each $m\in \N_0$ in the case $\Omega$ Lipschitz, dates back to Calder\'on \cite{Calderon61}. Stein \cite[p.~181]{Stein}, in the case $\Omega$ Lipschitz, constructed an extension operator $\mathcal{E}:W^m_2(\Omega)\to H^m(\R^n)$, not depending on $m\in \N_0$, that is continuous for all $m$. It is well known that if an open set is merely $C^{0,\beta}$, for some $\beta<1$, rather than $C^{0,1}$, then, for each $m\in\N$, there may exist no extension operator $\mathcal{E}:W^m_2(\Omega)\to H^m(\R^n)$, so that $H^m(\Omega)\subsetneqq W^m_2(\Omega)$. This is the case, for example, for the cusp domain in Lemma \ref{lem:segment} below: see \cite[p.~88]{Maz'ya}. (Here, as usual, {\em domain} means connected open set.)
A strictly larger class than the class of $C^{0,1}$ domains for which continuous extension operators do exist is the class of {\em locally uniform} domains \cite{Jones}.

\begin{defn} \label{jones} A domain $\Omega\subset \R^n$ is said to be {\em $(\epsilon,\delta)$ locally uniform} if, between any pair of points $x$, $y\in \Omega$ such that
$|x-y| < \delta$, there is a rectifiable arc $\gamma \subset \Omega$ of length at most $|x-y|/\epsilon$ and having the property
that, for all $z \in \gamma$,
$$
\dist(z, \partial \Omega) \geq \frac{\epsilon|z-x|\,|z-y|}{|x-y|}\, .
$$
\end{defn}

All Lipschitz domains are locally uniform, but the class of locally uniform domains contains also sets $\Omega\subset \R^n$ with wilder, fractal boundaries, in fact with $\partial \Omega$ having any Hausdorff
 dimension in $[n-1,n)$ \cite[p.~73]{Jones}.
Jones \cite{Jones} proves existence of an extension operator $\mathcal{E}:W^m_2(\Omega)\to H^m(\R^n)$ for each $m\in \N$ when $\Omega$ is locally uniform. More recently the following uniform result is proved.

\begin{thm}[Rogers, \cite{Rogers}] \label{rogers} If $\Omega\subset \R^n$ is an $(\epsilon,\delta)$ locally uniform domain then there exists an extension operator $\mathcal{E}:W^m_2(\Omega)\to H^m(\R^n)$, not depending on $m\in \N_0$, that is continuous for all $m$.
\end{thm}

The following uniform extension theorem for the spaces $H^s(\Omega)$ is a special case of a much more general uniform extension theorem for Besov spaces \cite{Rychkov}, and generalises Stein's classical result to negative~$s$. Rychkov's \cite{Rychkov} result is stated for Lipschitz hypographs and bounded Lipschitz domains, but his localisation arguments for bounded domains  \cite[p.~244]{Rychkov} apply equally to all Lipschitz open sets.

\begin{thm}[Rychkov, \cite{Rychkov}]\label{rychkov} If $\Omega\subset \R^n$ is a Lipschitz open set
or a Lipschitz hypograph,
then there exists an extension operator $\mathcal{E}:H^s(\Omega)\to H^s(\R^n)$, not depending on $s\in \R$, that is continuous for all~$s$.
\end{thm}

Combining Theorems \ref{rogers} and \ref{rychkov} with Lemma \ref{lem:ext} and Theorem \ref{simpler} (i) we obtain the following interpolation result.

\begin{cor} \label{cor:int} If $\Omega\subset \R^n$ is a Lipschitz open set or a Lipschitz hypograph, then $\{H^s(\Omega):s\in \R\}$ is an interpolation scale.
If $\Omega\subset \R^n$ is an $(\epsilon,\delta)$ locally uniform domain then $\{H^s(\Omega):s\geq 0\}$ is an interpolation scale.
\end{cor}

Except in the case $\Omega=\R^n$,  it appears that $\{H^s(\Omega):s\in \R\}$ is not an exact interpolation scale.
In particular, Lemma \ref{lem:segment} below shows that, for $\Omega = (0,a)$ with $0<a<1$,  $\{H^s(\Omega):0\leq s\leq 2\}$ is not an exact interpolation scale, indeed that, for interpolation between $L^2(\Omega)$ and $H^2(\Omega)$, the ratio of the interpolation norm to the intrinsic norm on $H^1(\Omega)$ can be arbitrarily small for small $a$.
Example \ref{ex:1Dfrac} below is a bounded open set $\Omega\subset \R$ for which
\begin{equation} \label{subnot}
H^1(\Omega) \subsetneqq \big(L^2(\Omega),H^2(\Omega)\big)_{1/2},
\end{equation}
so that $\{H^s(\Omega):0\leq s\leq 2\}$ is not an interpolation scale.
The following lemma exhibits \eqref{subnot} for a $C^{0,\beta}$  domain in $\R^2$, for every $\beta\in (0,1)$. These results contradict \cite[Theorem B.8]{McLean} which claims that $\{H^s(\Omega):s\in\R\}$ is an exact interpolation scale for any non-empty open $\Omega\subset\R^n$.
(The error in McLean's proof lies in the wrong deduction of the bound $K(t,U;Y)\le K(t,u;X)$ (in his notation) from $K(t,U;Y)^2\le \|u_0\|_{X_0}^2+t^2\|u_1\|_{X_1}^2$.)

\begin{lem}\label{lem:CuspNotIntp} For some $p > 1$ let $\Omega := \{(x_1,x_2)\in \R^2: 0<x_1<1 \mbox{ and } |x_2|< x_1^p\}$.
Then $\Omega$ is a $C^{0,\beta}$ domain for $\beta = p^{-1}<1$ and \eqref{subnot} holds, so that $\{H^s(\Omega):0\leq s\leq 2\}$ is not an interpolation scale.
\end{lem}
\begin{proof}
Let $\overline H_\theta:= (H^0(\Omega),H^2(\Omega))_\theta$, for $0<\theta<1$.
Choose an even function $\chi\in C^\infty(\R)$ such that $0\leq \chi(t)\leq 1$ for $t\in\R$, with $\chi(t)=0$ if $|t|>1$, and $\chi(t)=1$ if $|t|<1/2$.
For $0<h\leq 1$ define $\phi_h\in H^2(\Omega)$ by
$\phi_h(x) = \chi(x_1/h)$, $x\in \Omega$.
We observe that $\phi_h(x) = 0$ for $x_1>h$ so that, where $\Omega_h := \{x\in \Omega:0<x_1<h\}$,
$$
\|\phi_h\|_{L^2(\Omega)}^2 \leq \int_{\Omega_h} \rd x = 2\int_0^h x_1^p \rd x_1 = \frac{2h^{p+1}}{p+1} \, .
$$
Further, defining $\phi_h^+(x) = \chi(x_1/h)\chi(x_2/(2h))$, for $x=(x_1,x_2)\in \R^2$ and $0<h\leq 1$, it is clear that $\phi_h= \phi_h^+|_\Omega$.
Moreover, $\|\partial ^\alpha \phi_h^+\|_{L^2(\R^2)} = h^{1-|\alpha|}\|\partial ^\alpha \phi_1^+\|_{L^2(\R^2)}$, for $\alpha\in \N_0^2$.
Thus, using the identity \eqref{eq:normint},
$$
\|\phi_h\|_{H^2(\Omega)} \leq \|\phi^+_h\|_{H^2(\R^2)} = O(h^{-1}), \quad \mbox{as } h\to 0,
$$
so that, applying Lemma \ref{intnormub}(i), $\|\phi_h\|_{\overline H_\theta} = O(h^\beta)$, as $h\to 0$,
where $\beta := (1-\theta)(p+1)/2 - \theta$.
Let $\theta=1/2$, so that $\beta = (p-1)/4>0$.
Put $h_n = n^{-q}$, for $n\in\N$, for some $q>\beta^{-1}$.
Then $\|\phi_{h_n}\|_{\overline H_{1/2}} = O(n^{-q\beta})$ as $n\to\infty$, so that
$\sum_{n=1}^\infty \phi_{h_n}$
is convergent in $\overline H_{1/2}$ to some $\psi\in \overline H_{1/2}$.
Let $\Omega_b$ be some $C^1$ bounded domain  containing $\Omega$.
Then, by the Sobolev embedding theorem (e.g.~\cite[p.~97]{Adams}), $H^1(\Omega_b)\subset L^r(\Omega_b)$ for all $1\leq r<\infty$, so that $H^1(\Omega)\subset L^r(\Omega)$. We will show \eqref{subnot} by showing $\psi\not\in L^r(\Omega)$ if $r$ is sufficiently large.

Clearly, $\psi = \sum_{n=1}^\infty \phi_{h_n}$ satisfies
$\psi(x)\geq n$, for $0<x_1< h_n/2 = n^{-q}/2$,
so that
\begin{align*}
\int_\Omega |\psi|^r \rd x & \geq2\sum_{n=1}^\infty n^r \int_{(n+1)^{-q}/2}^{n^{-q}/2} x_1^p \rd x_1\\
 &= \frac{1}{(p+1)2^{p}}\sum_{n=1}^\infty n^r \left(n^{-q(p+1)}-(n+1)^{-q(p+1)}\right)\geq \frac{q}{2^{p}}\sum_{n=1}^\infty n^r(n+1)^{-q(p+1)-1},
\end{align*}
where in the last step we use the mean value theorem, which gives that, for some $\xi\in (n,n+1)$,
$n^{-t}-(n+1)^{-t} = t\xi^{-t-1}\geq t (n+1)^{-t-1}$, where $t=q(p+1)>0$.
Thus $\psi\not\in L^r(\Omega)$ if
$r-q(p+1)-1 \geq -1$, i.e., if $r \geq q(p+1)$.
Since we can choose any $q>\beta^{-1}$, we see that, in fact, $\overline H_{1/2} \not\subset L^r(\Omega)$ for $r>4(p+1)/(p-1)$.
\end{proof}

\subsection{The spaces \texorpdfstring{$\tH^s(\Omega)$}{tildeHs(Omega)}}
For $s\in\R$ and $\Omega\subset\R^n$ we define $\tH^s(\Omega):= \overline{\scrD(\Omega)}^{H^s(\R^n)}$, the closure of $\scrD(\Omega)$ in $H^s(\R^n)$.
We remark that if $\Omega$ is $C^0$ then $\tH^s(\Omega)=\{u\in H^s(\R^n):\supp{u}\subset \overline{\Omega} \}$ \cite[Theorem 3.29]{McLean}, but that these two spaces are in general different if $\Omega$ is not $C^0$ \cite{ChaHewMoi:13}.
Also, for any $m\in\N_0$, $\tH^m(\Omega)$ is unitarily isomorphic (via the restriction operator $\mathcal{R}$) to $H^m_0(\Omega)$, the closure of $\scrD(\Omega)$ in $H^m(\Omega)$.
For any open $\Omega\subset\R^n$, $\tH^s(\Omega)$ is a natural unitary realisation of the dual space of $H^{-s}(\Omega)$, with duality paring (cf.\ \cite[Theorem 3.14]{McLean})
\[
\langle u,v \rangle_{\tH^{s}(\Omega)\times {H}^{-s}(\Omega)}:= \langle u,V \rangle_{-s},
\quad \mbox{ for } u\in\tH^s(\Omega), \,v\in H^{-s}(\Omega),
\]
where $V\in H^{-s}(\R^n)$ denotes any extension of $v$ with $V|_\Omega=v$, and $\langle \cdot,\cdot \rangle_{-s}$ is the standard duality pairing on $H^s(\R^n)\times H^{-s}(\R^n)$, the natural extension of the duality pairing $\langle\cdot, \cdot\rangle$ on $\scrS(\R^n)\times \scrS^*(\R^n)$.
This result is well known when $\Omega$ is $C^0$ \cite{McLean}; that it holds for arbitrary $\Omega$ is shown in~\cite{ChaHewMoi:13,CoercScreen}.

The following corollary follows from this duality result and Theorem \ref{simpler} (ii).

\begin{cor} \label{cor:duall} Suppose that $s_0\leq s_1$, $0<\theta<1$, and set $s= s_0(1-\theta)+s_1\theta$, $\overline H = (H^{s_0}(\Omega), H^{s_1}(\Omega))$, and $\overline H^* = (\tH^{-s_0}(\Omega), \tH^{-s_1}(\Omega))$. Then $\overline H^*_\theta = (\tH^{-s_0}(\Omega), \tH^{-s_1}(\Omega))_\theta \subset \tH^{-s}(\Omega)$, with $\|u\|_{\tH^{-s}(\Omega)}\leq \|u\|_{\overline H^*_\theta}$, for $u\in \overline H^*_\theta$. Further, $\overline H_\theta = H^s(\Omega)$ if and only if $\overline H^*_\theta = \tH^{-s}(\Omega)$ and, if both these statements are true, then, for $a>1$,
$$
a^{-1} \|u\|_{H^s(\Omega)} \leq \|u\|_{\overline H_\theta},\; \forall u\in H^s(\Omega) \quad
\mbox{ if and only if } \quad \|u\|_{\overline H^*_\theta} \leq a\|u\|_{\tilde H^{-s}(\Omega)},\; \forall u\in \tH^{-s}(\Omega).
$$
\end{cor}
Combining this with  Corollary \ref{cor:int}, we obtain the following result.
\begin{cor} \label{cor:tildeint}
If $\Omega\subset \R^n$ is a Lipschitz open set or a Lipschitz hypograph, then $\{\tH^s(\Omega):s\in \R\}$ is an interpolation scale.
If $\Omega\subset \R^n$ is an $(\epsilon,\delta)$ locally uniform domain then $\{\tH^s(\Omega):s\leq 0\}$ is an interpolation scale.
\end{cor}

Except in the case $\Omega=\R^n$, it appears that $\{\tH^s(\Omega):s\in \R\}$ is not an exact interpolation scale. Example \ref{ex:tildeinterval} below shows, for the simple one-dimensional case $\Omega=(0,1)$, that $\{\tH^s(\Omega): 0\leq s\leq 1\}$ is not an exact interpolation scale, using a representation for the norm for interpolation between $L^2(\Omega)=\tH^0(\Omega)$ and $\tH^1(\Omega)$ given in the following lemma that illustrates the abstract Theorem \ref{thm:compact} (cf., \cite[Chapter 8]{Kress}). For the cusp domain example of Lemma \ref{lem:CuspNotIntp}, by Lemma \ref{lem:CuspNotIntp} and Corollary \ref{cor:duall}, $\{\tH^s(\Omega):-2\leq s\leq 0\}$ is not an interpolation scale at all.
\begin{lem}
\label{ex:tildeLipschitz}
Let $\Omega$ be bounded and set $H_0:=\tH^0(\Omega)=L^2(\Omega)$, $H_1:=\tH^1(\Omega)=H^1_0(\Omega)$. Then $\overline{H}:=(H_0,H_1)$ satisfies the assumptions of Theorem \ref{thm:compact}, since the embedding of $H_0^1(\Omega)$ into $L^2(\Omega)$ is compact. The orthogonal basis for $H_1$, $\{\phi_j:j\in \N\}$, of eigenvectors of $T$ (with $\lambda_j$ the eigenvalue corresponding to $\phi_j$ and $\|\phi_j\|_{H_0} = 1$), is a basis of eigenfunctions of the Laplacian. [This follows since $T\phi=\lambda \phi$, for $\lambda>0$ and $\phi\in H_1=H_0^1(\Omega)$, if and only if
\begin{equation} \label{eq:weak}
\int_{\Omega} \left(\nabla \phi \cdot \nabla \bar \psi - \rho \phi\bar\psi\right)\, \rd x=0, \quad \mbox{for } \psi\in H_0^1(\Omega),
\end{equation}
with $\rho = \lambda^{-1}-1$. In turn, by local elliptic regularity, \eqref{eq:weak} holds if and only if $\phi\in H_0^1(\Omega)\cap C^2(\Omega)$ and $-\Delta \phi = \rho \phi$ in $\Omega$ (in a
classical sense).]
From Theorem \ref{thm:compact}, the interpolation norm on $\overline{H}_s$ is
\begin{equation} \label{th:rangeLap}
\|\phi\|_{\overline{H}_s}=\|\phi\|_{s}^*
:= \bigg(\sum_{j=1}^\infty \lambda_j^{-s} |a_j|^2\bigg)^{1/2}
=\bigg(\sum_{j=1}^\infty (1+\rho_j)^{s} |a_j|^2\bigg)^{1/2}, \quad \mbox{for } 0<s<1 \mbox{ and } \phi\in \overline H_s,
\end{equation}
where, for $j\in \N$, $\rho_j := \lambda_j^{-1}-1$ and $a_j:= \int_\Omega \phi \bar \phi_j \rd x$. Further, $\|\phi\|_{\tH^j(\Omega)} = \|\phi\|_j^*$ for $\phi\in H_j = \tH^j(\Omega)$ and $j=0,1$. Moreover, by Corollary \ref{cor:tildeint}, if $\Omega$ is Lipschitz,
$\overline H_s = \tH^s(\Omega)$ for $0<s<1$, with equivalence of norms.
\end{lem}

\subsection{One-dimensional examples and counterexamples}
Our first example, Lemma~\ref{lem:segment}, which illustrates that $\{H^s(\Omega):0\leq s\leq 2\}$ needs not be an exact interpolation scale, requires explicit values for the $H^1(\Omega)$ and $H^2(\Omega)$ norms, for $\Omega = (0,a)$ with $a>0$.
These norms are computed using the minimal extension operator $\mathcal E_s:H^s(\Omega)\to H^s(\R)$ for $s=1,2$.
\begin{lem}\label{lem:H1H2}
For $\Omega=(0,a)\subset\R$, with $a>0$, the $H^1(\Omega)$ and $H^2(\Omega)$ norms are given by
\begin{align} \label{eq:H1norm0a}
\|\phi\|_{H^1(\Omega)}^2 =&
|\phi(0)|^2 + |\phi(a)|^2 + \int_{0}^a\left(|\phi|^2+ |\phi^\prime|^2\right) \,\rd x,\\
\|\phi\|_{H^2(\Omega)}^2=
&|\phi(0)|^2+|\phi'(0)|^2+|\phi(0)-\phi'(0)|^2+
|\phi(a)|^2+|\phi'(a)|^2+|\phi(a)-\phi'(a)|^2
\nonumber\\
&+\int_0^a (|\phi|^2+2|\phi'|^2+|\phi''|^2) \,\rd x.
\label{eq:H2norm0a}
\end{align}
\end{lem}
\begin{proof}
By the definitions \eqref{eq:Hsdef} and \eqref{eq:normint},
$\|\phi\|^2_{H^1(\Omega)}=\|\mathcal E_1\phi\|^2_{H^1(\R)}
=\int_\R (|\mathcal E_1\phi|^2+|(\mathcal E_1\phi)'|^2)\rd x$, where the extension $\mathcal E_1\phi$ of $\phi\in H^1(\Omega)$ with minimal $H^1(\R)$ norm is computed as an easy exercise in the calculus of variations, recalling that $H^1(\R)\subset C^0(\R)$, to be
\begin{align*}
\mathcal E_1 \phi(x)=\begin{cases}
\phi(0)\,\re^x, & x\le0,\\
\phi(x), & 0<x<a,\\
\phi(a)\,\re^{a-x}, &x\ge a.
\end{cases}\end{align*}
The assertion~\eqref{eq:H1norm0a} follows by computing $\int_\R (|\mathcal E_1\phi|^2+|(\mathcal E_1\phi)'|^2)\rd x$.

Similarly, for $\phi\in H^2(\Omega)$, $\|\phi\|_{H^2(\Omega)}= \|\mathcal{E}_2\phi\|_{H^2(\R)}$ and $\mathcal E_2\phi$ is computed by minimizing the functional
$$\mathcal J_2(\psi)=\|\psi\|^2_{H^2(\R)}=\int_\R(1+\xi^2)^2|\hat \psi|^2\rd\xi
= \int_\R (|\psi|^2+2|\psi'|^2+|\psi''|^2)\rd x$$
under the constraint $\psi|_\Omega=\phi$.
By computing the first variation of the functional $\mathcal J_2$ and integrating by parts, we see that $\psi$ solves the differential equation $\psi''''-2\psi''+\psi=0$ (whose solutions are $e^x,e^{-x},xe^x,xe^{-x}$) in the complement of $\Omega$, and, recalling that $H^2(\R)\subset C^1(\R)$, we obtain
\begin{align*}\mathcal E_2\phi(x) =
\begin{cases}
xe^x\phi'(0)+(1-x)e^x\phi(0), & x\le 0,\\
\phi(x), & 0<x<a,\\
(x-a)e^{a-x}\phi'(a)+(1-a+x)e^{a-x}\phi(a), & x\ge a.
\end{cases}
\end{align*}
The assertion \eqref{eq:H2norm0a} is obtained by computing  $\int_\R (|\mathcal E_2\phi|^2 + 2|(\mathcal E_2\phi )'|^2+|(\mathcal E_2\phi )''|^2)\rd x$.
\end{proof}

\begin{lem} \label{lem:segment} If $\Omega=(0,a)$, with $a>0$, then $\{H^s(\Omega):0\leq s\leq 2\}$ is not an exact interpolation scale. In particular, where $\overline H_\theta := (L^2(\Omega),H^2(\Omega)_\theta$, it holds that $\overline H_\theta = H^{2\theta}(\Omega)$, for $0<\theta<1$, but
$\|1\|_{\overline H_{1/2}} \neq \|1\|_{H^{1}(\Omega)}$. Precisely,
\begin{equation} \label{ratio}
\frac{\|1\|_{\overline H_{1/2}}}{\|1\|_{H^{1}(\Omega)}} \leq \left(\frac{a^2+4a}{a^2+4a+4}\right)^{1/4} < \min(a^{1/4},1).
\end{equation}
\end{lem}
\begin{proof}
The inequality \eqref{ratio} follows from Lemma~\ref{lem:H1H2} and Lemma~\ref{intnormub}(i), which give that
\begin{align*}
\|1\|^2_{L^2(\Omega)}= a,\;\;
\|1\|^2_{H^1(\Omega)}= 2+a,\;\;
\|1\|^2_{H^2(\Omega)}= 4+a,\;\;
\|1\|^2_{\overline H_{1/2}(\Omega)}
\le \|1\|_{L^2(\Omega)}\|1\|_{H^2(\Omega)}
=\sqrt{a^2+4a}.
\end{align*}

\vspace*{-3ex}
\end{proof}

Lemma \ref{lem:segment} shows that, for the regular domain $\Omega=(0,a)$, the spaces $H^s(\Omega)$ are not an \emph{exact} interpolation scale, and that the ratio \eqref{ratio} between the interpolation norm and the $H^s(\Omega)$ norm can be arbitarily small. (However, the two norms are equivalent:
Corollary \ref{cor:int} shows that $H^s(\Omega)$ constitutes an interpolation scale in this case.)
The next example provides an irregular open set for which $\{H^s(\Omega):0\le s\le2\}$ is {\em not} an interpolation scale so that, by Corollary \ref{cor:duall}, also $\{\tH^s(\Omega):-2\le s\le0\}$ is not an interpolation scale

\begin{example} \label{ex:1Dfrac}
Let $\underline a:=(a_1, a_2, \ldots)$ be a real sequence satisfying
$a_1 := 1$, $0< a_{n+1} < a_n/4$, $n\in\N$,
and let
$\Omega := \bigcup_{n=1}^\infty (a_n/2,a_n)\subset(0,1)$.
Let $H_0:=L^2(\Omega)$, $H_1:=H^2(\Omega)$, $\overline H:=(H_0,H_1)$, and $\overline H_{1/2}:=(H_0,H_1)_{1/2}$.
We note first that if $u\in H^1(\R)$ then, by standard Sobolev embedding results \cite[p.~97]{Adams}, $u\in C^0(\R)$, so $u|_\Omega\in L^\infty(\Omega)$ and $H^1(\Omega)\subset L^\infty(\Omega)$.
We will see that there is a choice of the sequence $\underline a=(a_1,a_2,...)$ such that $\overline H_{1/2}\not\subset L^\infty(\Omega)$ so that $\overline H_{1/2}\neq H^1(\Omega)$.

To see this, choose an even function $\chi\in C^\infty(\R)$ such that $\chi(t) = 0$ for $|t|>1$ and $\chi(0)=1$,
and consider the sequence of functions in $H_1 \subset \overline H_{1/2}\cap H^1(\Omega)$ defined by
$$
\phi_n(t) = \left\{\begin{array}{cc}
1, & t\in [0,a_n]\cap \Omega, \\
0, & t\in (a_n,\infty) \cap \Omega,\end{array}\right.
$$
for $n\in\N$. Clearly
$$
\|\phi_n\|_{H_0} \leq a_n^{1/2}\qquad \text{and}\qquad
\|\phi_n\|_{H_1} = \inf_{\psi\in H^2(\R), \, \psi|_\Omega = \phi_n} \|\psi\|_{H^2(\R)} \leq \|\psi_n\|_{H^2(\R)},
$$
where $\psi_n\in C^1(\R)\cap H^2(\R)$ is defined by
$$
\psi_n(t) = \left\{\begin{array}{cc}
\chi(t), & t<0, \\
1, & 0\leq t\leq a_n, \\
\chi((t-a_n)/b_n), & t>a_n,   \end{array}\right.
$$
with $b_1:=1$ and $b_n := a_{n-1}/2-a_{n}$, for $n\geq 2$.
Further, where $\alpha:=\|\chi\|_{H^2(\R)}$,
$$
\|\psi_n\|^2_{H^2(\R)} = \int_{-\infty}^\infty \left(|\psi_n|^2 + 2|\psi^\prime_n|^2 + |\psi^{\prime\prime}_n|^2\right) \, \rd t
= \frac{\alpha}{2} + a_n + \int_{0}^\infty \left(b_n|\chi(r)|^2 + 2b_n^{-1}|\chi^\prime(r)|^2 + b_n^{-3}|\chi^{\prime\prime}(r)|^2\right) \, \rd r,
$$
and, for $n\geq 2$, $a_n\leq 1/2$, $1/2\geq b_{n} \geq a_{n-1}/4$, so that
$$
\|\psi_n\|^2_{H^2(\R)} \leq \frac{1}{2} \left(1 + \left(1+b_n^{-3}\right)\alpha\right)
\leq \frac{1}{2} \left(1 + \left(1+64 a_{n-1}^{-3}\right)\alpha\right).
$$
Applying Lemma \ref{intnormub}(i) we see that, for $n\geq 2$,
$$
\|\phi_n\|_{\overline H_{1/2}} \leq \|\phi_n\|_{H_0}^{1/2}\|\phi_n\|_{H_1}^{1/2} \leq 2^{-1/4}a_n^{1/4} \left(1 + \left(1+64 a_{n-1}^{-3}\right)\alpha\right)^{1/4}.
$$
Now choosing $a_n$ according to the rule
$$
a_1 = 1, \quad a_n = \frac{a_{n-1}}{4}\left(1 + \left(1+64 a_{n-1}^{-3}\right)\alpha\right)^{-1} <  \frac{a_{n-1}}{4}, \quad n = 2,3,\ldots,
$$
it follows that $a_n \leq 4^{-n}$ and that $\|\phi_n\|_{\overline H_{1/2}} \leq 2^{-1/4}4^{-n/4}
\le (\sqrt{2})^{-n}\to 0$ as $n\to\infty$.
In fact $\phi_n\to 0$ so rapidly that
$\sum_{n=1}^\infty \phi_n$
is convergent in $\overline H_{1/2}$ 
to a limit $\Phi\in \overline H_{1/2}$.
This limit is not in $H^1(\Omega)$ as $\Phi\not\in L^\infty(\Omega)$:
explicitly, $\Phi(t) = n$, for $a_n/2<t<a_n$.
\end{example}

Our last example uses the results of Lemma \ref{ex:tildeLipschitz}, and shows that $\{\tH^s(0,1):0\leq s\leq 1\}$ is not an exact interpolation scale by computing values of the Sobolev and interpolation norms for specific functions.
This example also demonstrates that no normalisation 
of the interpolation norm can make the two norms equal.
\begin{example}
\label{ex:tildeinterval}
Let $\Omega=(0,1)$, $H_0=\tH^0(\Omega)=L^2(\Omega)$ and $H_1=\tH^1(\Omega)=H^1_0(\Omega)$.
The eigenfunctions and eigenvalues in Lemma \ref{ex:tildeLipschitz} are $\phi_j(x)=\sqrt{2}\sin(j\pi  x)$ and $\rho_j = j^2\pi^2$, so that, for $0<\theta<1$, the interpolation norm on $\overline H_\theta=\tH^\theta (\Omega)$ is given by \eqref{th:rangeLap}. In particular,
\begin{align*}\label{}
\|\phi_j\|^*_\theta = (1+ j^2\pi^2)^{\theta/2}, \quad \mbox{for }j\in \N.
\end{align*}
Noting that
\begin{align*}
\label{}
\hat{\phi_j}(\xi) = \frac{1}{\sqrt{\pi}}\int_{0}^1 \sin(j\pi x) \re^{-\ri \xi x} \,\rd x
= \frac{j\sqrt{\pi}\left(1-(-1)^j\re^{-\ri\xi} \right)}{j^2\pi^2-\xi^2}
= \frac{2j\sqrt{\pi}\re^{-\ri \xi/2}}{j^2\pi^2-\xi^2}
\begin{cases}
\cos\xi/2, & j \textrm{ odd}, \\
\ri\sin\xi/2, & j \textrm{ even},
\end{cases}
\end{align*}
it holds that
\begin{align*}\label{}
\|\phi_j\|_{\tH^\theta(\Omega)}
=\left(\int_{-\infty}^\infty (1+\xi^2)^\theta|\hat{\phi_j}(\xi)|^2 \,\rd \xi\right)^{1/2}
=2j\sqrt{2\pi}\left(\int_{0}^\infty \frac{(1+\xi^2)^\theta}{(j^2\pi^2-\xi^2)^2}\left\lbrace \begin{array}{ll}
\cos^2(\xi/2)\\
\sin^2(\xi/2)
\end{array}\right\rbrace \,\rd \xi\right)^{1/2}.
\end{align*}
A comparison of $\|\phi_j\|^*_{\theta}$ and $\|\phi_j\|_{\tH^\theta(\Omega)}$ for $j=1,2$ and $\theta\in(0,1)$ is shown in Figure \ref{fig:norms}(a). It is clear from Figure \ref{fig:norms}(a) that the interpolation and Sobolev norms do not coincide in this case. In particular, for $\theta=1/2$ we have
\begin{align*}\label{}
\|\phi_1\|^*_{1/2}\approx 1.816, \quad
\|\phi_1\|_{\tH^{1/2}(\Omega)}\approx 1.656, \qquad
\|\phi_2\|^*_{1/2} \approx 2.522, \quad
\|\phi_2\|_{\tH^{1/2}(\Omega)} \approx   2.404.
\end{align*}

\begin{figure}[!t] 
\subfigure{\includegraphics[width=80mm]{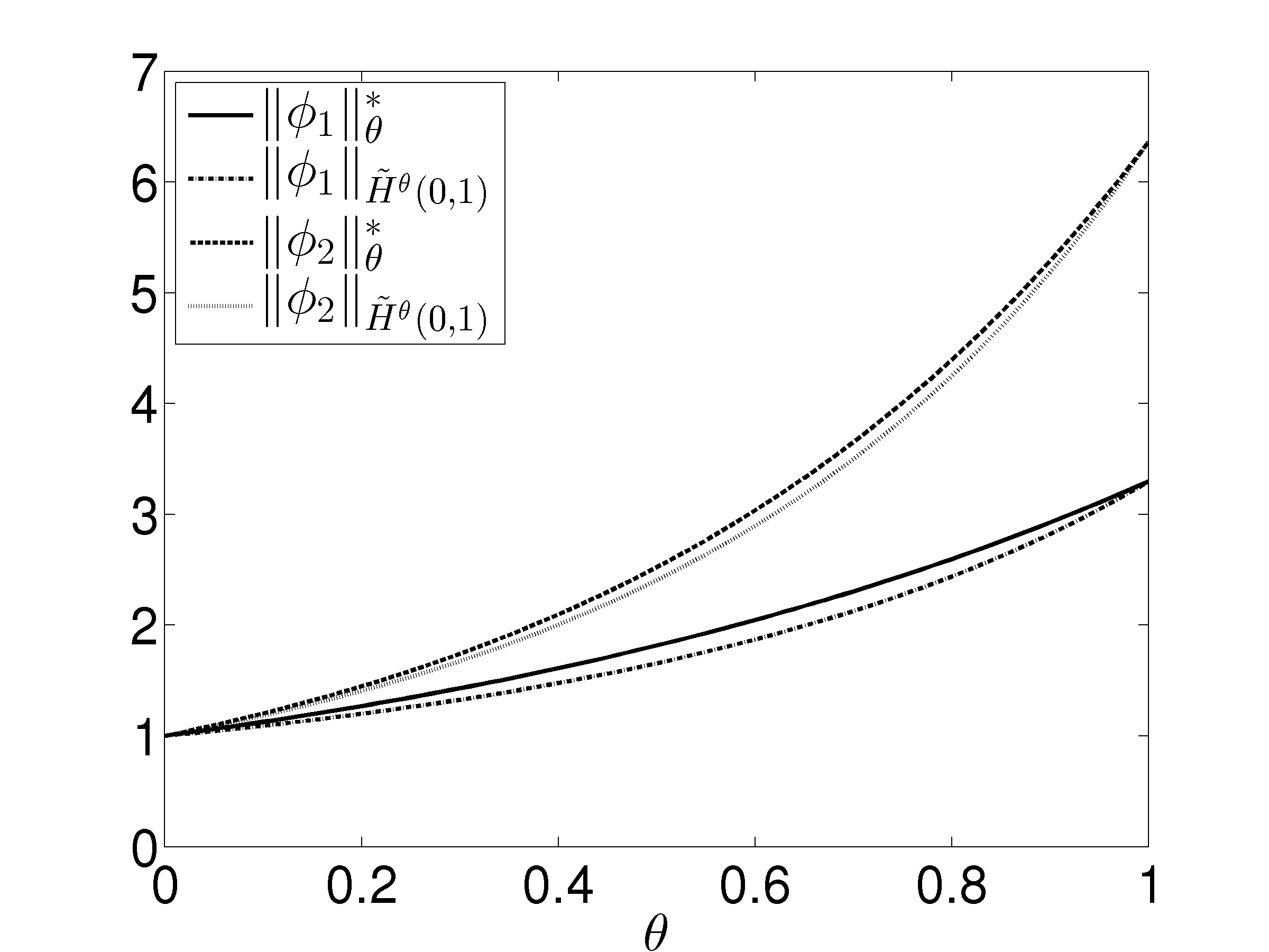}} \hs{-10}
\subfigure{\includegraphics[width=80mm]{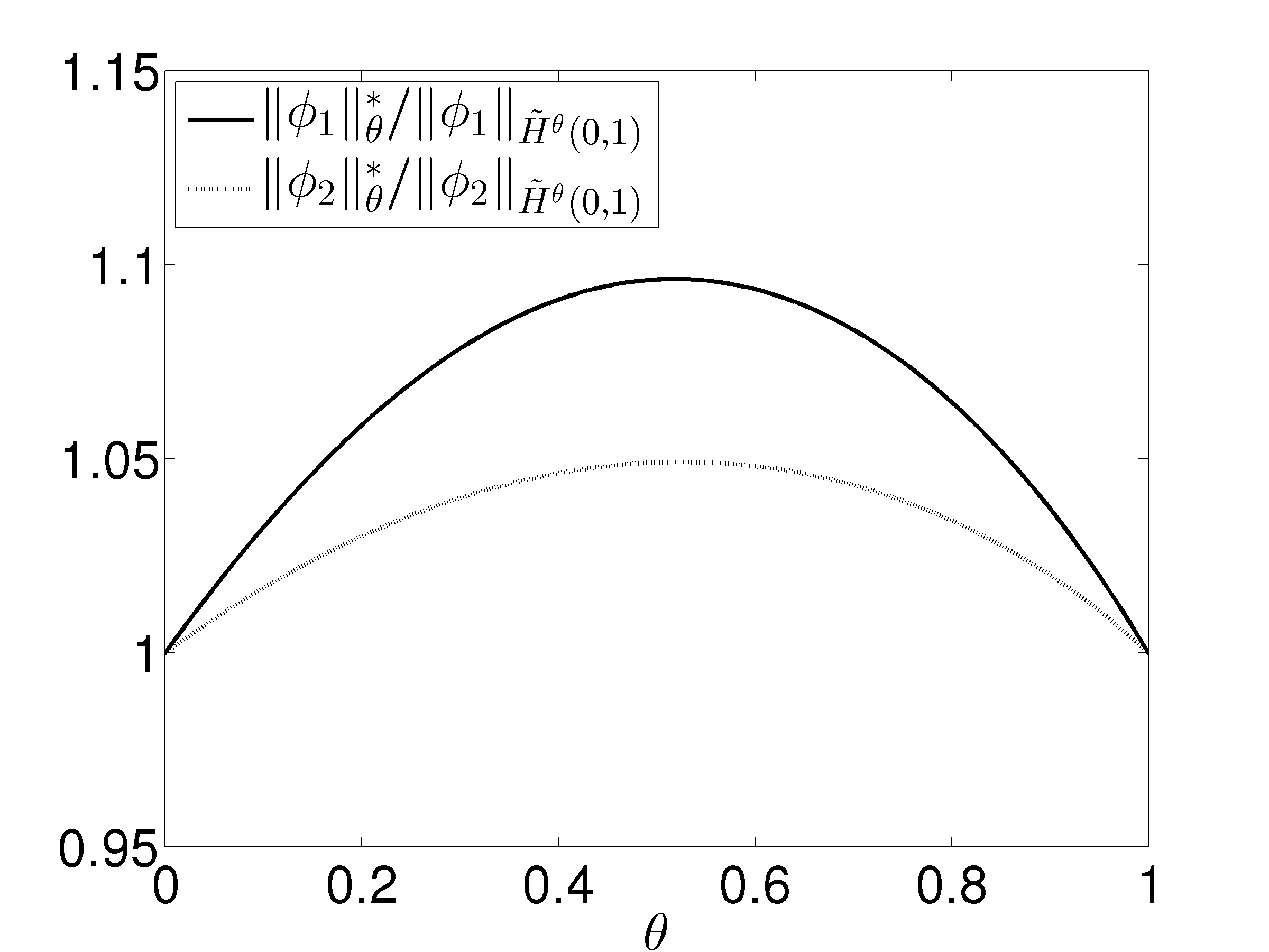}}
\caption{Comparison of Sobolev and interpolation norms in $\tH^\theta (\Omega)$, 
 for the functions $\phi_1$ and $\phi_2$ of Example \ref{ex:tildeinterval}.}
\label{fig:norms}
\end{figure}

The ratio between the two norms is plotted for both $\phi_1$ and $\phi_2$ in Figure \ref{fig:norms}(b). In particular,
\begin{align*}\label{}
\|\phi_1\|^*_{1/2}/\|\phi_1\|_{\tH^{1/2}(\Omega)}\approx 1.096, \qquad
\|\phi_2\|^*_{1/2}/\|\phi_2\|_{\tH^{1/2}(\Omega)} \approx 1.049.
\end{align*}
 As the values of these two ratios are different, not only are the two norms not equal with the normalisation \eqref{eq:Nchoice} we have chosen, it is clear that there is no possible choice of normalisation factor in the definition \eqref{eq:normK} that could make the interpolation and Sobolev norms equal.
\end{example}


\begin{thebibliography}{10}

\bibitem{Adams}
{\sc R.~A. Adams}, {\em Sobolev Spaces}, Academic Press, 1973.

\bibitem{Ameur2004}
{\sc Y.~Ameur}, {\em A new proof of {D}onoghue's interpolation theorem}, J.
  Funct. Spaces Appl., 2 (2004), pp.~253--265.

\bibitem{Ameur2014}
\leavevmode\vrule height 2pt depth -1.6pt width 23pt, {\em Interpolation and
  operator constructions}. Preprint, 2014, \href{https://arxiv.org/abs/1401.6090}{arXiv:1401.6090} (accessed 20/10/2014). 

\bibitem{BeSh}
{\sc C.~Bennet and R.~Sharpley}, {\em Interpolation of Operators}, Academic
  Press, 1988.

\bibitem{BeLo}
{\sc J.~Bergh and J.~L{\"{o}}fstr{\"{o}}m}, {\em Interpolation Spaces: an
  Introduction}, Springer-Verlag, 1976.

\bibitem{Bramble}
{\sc J.~H. Bramble}, {\em Multigrid Methods}, Chapman \& Hall, 1993.

\bibitem{Calderon61}
{\sc A.~P. Calder{\'o}n}, {\em Lebesgue spaces of differentiable functions and
  distributions}, Proc. Symp. Pure Math, 4 (1961), pp.~33--49.

\bibitem{CoercScreen}
{\sc S.~N. Chandler-Wilde and D.~P. Hewett}, {\em Acoustic scattering by
  fractal screens: mathematical formulations and wavenumber-explicit continuity
  and coercivity estimates}.
\newblock University of Reading Preprint, 2013, MPS-2013-17; \href{https://arxiv.org/abs/1401.2805}{arXiv:1401.2805} (accessed 20/10/2014).

\bibitem{ChaHewMoi:13}
{\sc S.~N. Chandler-Wilde, D.~P. Hewett, and A.~Moiola}, {\em Sobolev spaces on
  subsets of $\mathbb{R}^n$ with application to boundary integral equations on
  fractal screens}.
\newblock {I}n preparation.

\bibitem{Donoghue}
{\sc W.~Donoghue}, {\em The interpolation of quadratic norms}, Acta
  Mathematica, 118 (1967), pp.~251--270.

\bibitem{DSII}
{\sc N.~Dunford and J.~T. Schwarz}, {\em Linear Operators, Part II. Spectral
  Theory}, John Wiley, 1963.

\bibitem{Jones}
{\sc P.~W. Jones}, {\em Quasiconformal mappings and extendability of functions
  in {S}obolev spaces}, Acta Mathematica, 147 (1981), pp.~71--88.

\bibitem{Kato}
{\sc T.~Kato}, {\em Perturbation Theory for Linear Operators}, Springer-Verlag,
  1980.
\newblock 2nd Edition.

\bibitem{Kress}
{\sc R.~Kress}, {\em Linear Integral Equations}, Springer-Verlag, New York, 2nd
  ed., 1999.

\bibitem{LiMaI}
{\sc J.-L. Lions and E.~Magenes}, {\em Non-Homogeneous Boundary Value Problems
  and Applications I}, Springer-Verlag, 1972.

\bibitem{Maz'ya}
{\sc V.~G. Maz'ya}, {\em Sobolev Spaces with Applications to Elliptic Partial
  Differential Equations}, Springer, 2nd~ed., 2011.

\bibitem{McCarthy}
{\sc J.~E. McCarthy}, {\em Geometric interpolation between {H}ilbert spaces},
  Arkiv f{\"o}r Matematik, 30 (1992), pp.~321--330.

\bibitem{McLean}
{\sc W.~McLean}, {\em Strongly Elliptic Systems and Boundary Integral
  Equations}, CUP, 2000.

\bibitem{Peetre}
{\sc J.~Peetre}, {\em A theory of interpolation of normed spaces}, Notas de
  Matem\'atica, No. 39, Instituto de Matem\'atica Pura e Aplicada, Conselho
  Nacional de Pesquisas, Rio de Janeiro, 1968.

\bibitem{Rogers}
{\sc L.~G. Rogers}, {\em Degree-independent {S}obolev extension on locally
  uniform domains}, J. Funct. Anal., 235 (2006), pp.~619--665.

\bibitem{Rychkov}
{\sc V.~S. Rychkov}, {\em On restrictions and extensions of the {B}esov and
  {T}riebel--{L}izorkin spaces with respect to {L}ipschitz domains}, J. London
  Math. Soc., 60 (1999), pp.~237--257.

\bibitem{Stein}
{\sc E.~M. Stein}, {\em Singular integrals and differentiability properties of
  functions}, vol.~1, Princeton University Press, 1970.

\bibitem{Tartar}
{\sc L.~Tartar}, {\em An introduction to Sobolev spaces and interpolation
  spaces}, Springer-Verlag, 2007.

\bibitem{Triebel78ITFSDO}
{\sc H.~Triebel}, {\em Interpolation Theory, Function Spaces, Differential
  Operators}, North Holland, 1978.

\end{thebibliography}

\end{document}